\documentclass[oneside,english]{amsart}
\usepackage[T1]{fontenc}
\usepackage[latin9]{inputenc}
\usepackage{geometry}
\usepackage{color}
\usepackage{babel}
\usepackage{float}
\usepackage{textcomp}
\usepackage{amstext}
\usepackage{amsthm}
\usepackage{amssymb}
\usepackage{graphicx}
\usepackage{setspace}
\usepackage{esint}
\PassOptionsToPackage{normalem}{ulem}
\usepackage{ulem}
\usepackage[unicode=true,pdfusetitle,
 bookmarks=true,bookmarksnumbered=false,bookmarksopen=false,
 breaklinks=false,pdfborder={0 0 0},pdfborderstyle={},backref=false,colorlinks=true]
 {hyperref}
\hypersetup{
 citecolor=blue}

\makeatletter
\numberwithin{equation}{section}
\numberwithin{figure}{section}
\theoremstyle{plain}
\newtheorem{thm}{\protect\theoremname}[section]
\theoremstyle{definition}
\newtheorem{defn}[thm]{\protect\definitionname}
\theoremstyle{remark}
\newtheorem{rem}[thm]{\protect\remarkname}
\theoremstyle{plain}
\newtheorem{prop}[thm]{\protect\propositionname}
\theoremstyle{plain}
\newtheorem{lem}[thm]{\protect\lemmaname}
\theoremstyle{plain}
\newtheorem{cor}[thm]{\protect\corollaryname}

\theoremstyle{lemma}
\usepackage{wrapfig}
\usepackage{thmtools}
\usepackage{thm-restate}

\makeatother

\providecommand{\corollaryname}{Corollary}
\providecommand{\definitionname}{Definition}
\providecommand{\lemmaname}{Lemma}
\providecommand{\propositionname}{Proposition}
\providecommand{\remarkname}{Remark}
\providecommand{\theoremname}{Theorem}

\begin{document}
\title{Recurrence of a Weighted Random Walk on a Circle Packing with Parabolic
Carrier}
\author{Ori Gurel-Gurevich and Matan Seidel}
\begin{abstract}
In this paper we show that given a circle packing of an infinite planar
triangulation such that its carrier is parabolic, placing weights
on the edges according to a certain natural way introduced by Dubejko,
makes the random walk recurrent. We also propose a higher-dimensional
analogue of the Dubejko weights.
\end{abstract}

\maketitle

\section{Introduction}

A \textbf{circle packing} is a collection of circles in the plane
with disjoint interiors. The \textbf{tangency graph} of a circle packing
is the graph obtained by assigning a vertex to each circle and connecting
two vertices by an edge if their respective circles are tangent to
one another. A planar graph is called a \textbf{triangulation }if
it admits a drawing in the plane (also called a triangulation) in
which all faces are incident to exactly 3 edges, outer face included.
The celebrated circle packing theorem \cite{koebe1936kontaktprobleme,stephenson2005introduction}
asserts that every finite planar graph is the tangency graph of some
circle packing. Furthermore, if the graph is a triangulation then
its circle packing is unique up to Möbius transformations and reflections.
A concise background on the probabilistic and combinatorial properties
of circle packings can be found in \cite{nachmias2018planar}.

Infinite planar graphs can also be shown (see \cite{nachmias2018planar})
to be isomorphic to the tangency graph of some (infinite) circle packing.
However, the question of uniqueness becomes more complicated, and
requires a few more definitions. All infinite graphs in this paper
are assumed to be connected and locally finite, and infinite triangulations
are assumed to have no outer face. For an infinite triangulation drawn
in the plane we use the term \emph{face} to also mean the compact
set bounded by its edges. A circle packing induces a drawing in straight
lines of its tangency graph by mapping the vertices to the centers
of their corresponding circles. If this drawing is an infinite triangulation
and the union of its faces is a domain $\Omega\subseteq\mathbb{R}^{2}$,
the circle packing is said to be a \textbf{circle-packed infinite
triangulation of }$\Omega$, $\Omega$ is called its \textbf{carrier}
and the tangency graph is said to be \textbf{circle-packable in $\Omega$}.

In \cite{he1993fixed} and \cite{he1995hyperbolic}, He \& Schramm
extended the circle packing theorem to infinite triangulations that
are \textbf{one-ended}, i.e. such that the removal of any finite set
of vertices leaves the graph with exactly one infinite connected component
(and possibly more finite ones). They showed that the possible carriers
of a circle packing of such a triangulation are deeply linked to properties
of the simple random walk on the graph. A graph is said to be \textbf{recurrent}
if the simple random walk started at some vertex $\rho$ almost surely
returns to $\rho$ infinitely often, and \textbf{transient} otherwise.
Indeed, among their results, they showed that for a bounded-degree
one-ended triangulation, either the graph is recurrent and is circle-packable
in the plane, or it is transient and circle-packable in the open unit
disk. Another of their results shows that in the transient case, the
graph is also circle-packable in the open unit square (or any other
simply-connected domain strictly contained in the plane). Consequently,
since Möbius transformations and reflections cannot map the unit disk
to the unit square, one cannot hope for the same rigidity as in the
finite case.

When removing the assumption of one-endedness, some more definitions
are needed: A domain $\Omega\subseteq\mathbb{R}^{2}$ is called \textbf{parabolic}
if for any open set $U\subseteq\Omega$, Brownian motion started at
any point of $\Omega$ and killed at $\partial\Omega$ hits $U$ almost
surely. The domain is said to be \textbf{hyperbolic} otherwise. An
equivalent formulation for parabolicity is given in Proposition \ref{prop: Equivalent Condition for Parabolicity }.
In \cite{gurel2017recurrence}, Gurel-Gurevich, Nachmias \& Souto
showed that a dichotomy still holds without one-endedness for infinite
triangulations of bounded degree: the graph is recurrent iff the carrier
of any circle packing of it is parabolic. Their proof relied on the
Rodin-Sullivan \cite{rodin1987convergence} Ring Lemma, which shows
that when the degree is bounded, the radii of adjacent circles must
be comparable in length. Indeed, removing the bounded degree assumption
may cause the theorem to fail. For example, as described in \cite{nachmias2018planar},
we can add circles to the circle packing of the (recurrent) hexagonal
lattice in a way that creates a drift in the random walk in a direction
of choice, rendering it transient. However, the carrier has remained
parabolic, being the entire plane.
\begin{defn}
\label{def: network}A \textbf{network} is a pair $\left(G,c\right)$
where $G=\left(V,E\right)$ is a connected graph (with self-loops
allowed) and $c:E\rightarrow\left(0,\infty\right)$ is a weight function
on the edges. The \textbf{weighted random walk} on a network is the
Markov chain with state space $V$ and transition probabilities $P_{x,y}=\frac{c_{xy}}{\pi\left(x\right)}$
where $\pi\left(x\right)=\sum_{y\sim x}c_{xy}$. The network is said
to be \textbf{recurrent} if the weighted random walk started at some
vertex $\rho\in V$ almost surely returns to $\rho$ infinitely often,
and \textbf{transient} otherwise.
\end{defn}

In \cite{dubejko1997random}, Dubejko proposed a way to place weights
on the edges of a circle-packed infinite triangulation such that in
the weighted random walk on the network, the sequence of centers of
the circles visited becomes a martingale. These weights also arise
naturally in the context of discrete complex analysis (see \cite{duffin1968potential}).
For completeness, we provide Dubejko's elegant proof here as Theorem
\ref{Theorem: Dubejko Weights are a Martingale}. Let us precisely
define the weights. In the straight-line drawing induced by a circle-packed
infinite triangulation, each face $f$ is a straight-edge triangle.
Thus, a circle packing for the dual graph is induced by mapping each
face $f$ to its incircle. In the drawings of the graph and its dual,
an edge $e$ and its dual $e^{\dagger}$ are orthogonal straight lines,
as shown in Figure \ref{fig: drawing for definition of Dubejko weights}:\\
\\
\begin{figure}[H]
\includegraphics[scale=0.75]{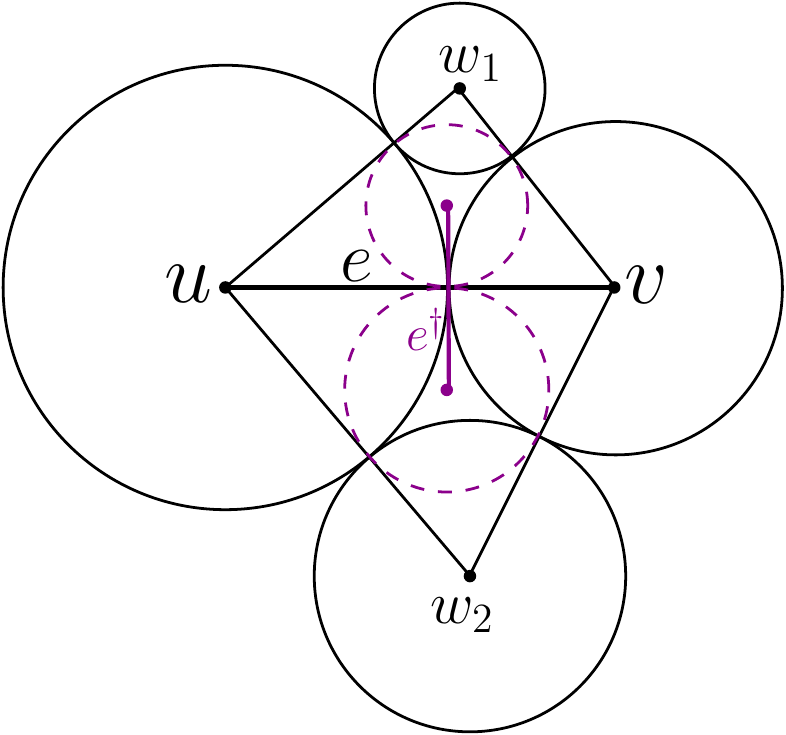}\caption{Part of a circle-packed triangulation in black and the induced circle
packing of the dual graph in purple. The dual edge $e^{\dagger}$
of $e=uv$ connects the incenters of the two faces incident to $e$.\label{fig: drawing for definition of Dubejko weights}}
\end{figure}
\begin{defn}
Let $\left\{ C_{v}\right\} _{v\in V}$ be a circle-packed infinite
triangulation. The \textbf{Dubejko weight} of an edge $e\in E$ is
$c_{e}=\frac{\left|e^{\dagger}\right|}{\left|e\right|}$, where $\left|e\right|$
and $\left|e^{\dagger}\right|$ are the respective lengths of the
straight line segments that $e$ and its dual edge $e^{\dagger}$
are mapped to.
\end{defn}

\noindent Proposition \ref{claim:weightsBoundedAbove} shows that
the weights are bounded from above by a constant. Thus, Rayleigh monotonicity
implies that if the simple random walk on a circle-packed triangulation
is recurrent then so is the weighted random walk (WRW). Furthermore,
as shown in Proposition \ref{claim: weights bounded from below if bounded degree},
when the graph has bounded degree the weights are also bounded from
below by a constant, and so in the bounded degree case the SRW and
WRW are either both recurrent or both transient. So one might hope
that replacing the SRW with the WRW would make the theorem of Gurel-Gurevich,
Nachmias \& Souto hold for both the bounded degree and unbounded degree
cases. The main goal of this paper is to prove the following:

\begin{restatable}{thm}{mainthm}\label{thm: main theorem-1}Let $\Omega\subseteq\mathbb{R}^{2}$ be a parabolic domain. Then for any circle-packed infinite triangulation of $\Omega$, the Dubejko-weighted random walk is recurrent.\end{restatable}

\noindent In section \ref{sec: Some-Geometric-Lemmas} we analyze
some geometric properties of the weights. In section \ref{sec:Special-Cases}
we prove two special cases: that if the carrier is the entire plane
then the WRW is recurrent, and that if the carrier is bounded then
the WRW is transient. In section \ref{sec: Some-Integration-Lemmas}
we prove some lemmas on integration of harmonic functions needed for
the proof of Theorem \ref{thm: main theorem-1}, which is proven itself
in section \ref{sec:The-General-Case}. Finally, in section \ref{sec:Extension-to-higher},
we propose an analogue of the Dubejko weights for higher dimensions,
and prove they make the weighted random walk into a martingale.

\section{Some Geometric Lemmas\label{sec: Some-Geometric-Lemmas}}
\begin{rem}
In the context of a circle packing, we use interchangeably the vertices
of the tangency graph and the centers of their corresponding circles
in $\mathbb{R}^{2}$.
\end{rem}

\begin{defn}
Let $\left\{ C_{v}\right\} _{v\in V}$ be a circle-packing of an infinite
triangulation. For a vertex $v\in V$, the \textbf{polygon of $v$}
denoted by $P_{v}$ is the polygon whose sides are the dual edges
to all the edges incident to $v$.
\end{defn}

\begin{rem}
$P_{v}$ is convex, its vertices are the incenters of the faces incident
to $v$ and the circle $C_{v}$ is inscribed in it.
\end{rem}

\noindent We now provide our main probabilistic motivation for looking
at the Dubejko weights, taken from \cite{dubejko1997random}. An analogue
for higher dimensions of this theorem is proven in Proposition \ref{prop:higher dimensions martingale theorem}.
\begin{thm}
\label{Theorem: Dubejko Weights are a Martingale}Let $\left\{ C_{v}\right\} _{v\in V}$
be a circle-packed infinite triangulation. Let $\left(Z_{n}\right)_{n\in\mathbb{N}}$
be the sequence of vertices visited during a Dubejko-weighted random
walk on $\left\{ C_{v}\right\} _{v\in V}$. Then $\left(Z_{n}\right)_{n\in\mathbb{N}}$
is a martingale.
\end{thm}

\begin{proof}
Set $\pi\left(x\right)=\sum_{y\sim x}c_{xy}$ and denote the transition
probabilities by $p_{xy}=\frac{c_{xy}}{\pi\left(x\right)}$. Let $v\in V$,
and let $u_{1},u_{2},...,u_{n}\in V$ be its neighbors in $G$. For
each $i\in\left\{ 1,2,...,n\right\} $ let $\overrightarrow{e_{i}}=u_{i}-v$
be the edge $vu_{i}$ oriented from $v$ to $u_{i}$. Let $R:\mathbb{R}^{2}\rightarrow\mathbb{R}^{2}$
be the linear clockwise rotation by $\frac{\pi}{2}$ radians, and
let $\overrightarrow{f_{i}}$ be the dual edge of $\overrightarrow{e_{i}}$
oriented in a direction such that $R\left(\widehat{e_{i}}\right)=\widehat{f_{i}}$,
where $\widehat{e_{i}}$ and $\widehat{f_{i}}$ are the respective
unit vectors of $\overrightarrow{e_{i}}$ and $\overrightarrow{f_{i}}$.\\
It is enough to show that $\sum_{i=1}^{n}p_{vu_{i}}\overrightarrow{e_{i}}=0$.
We have:
\begin{align*}
R\left(\sum_{i=1}^{n}p_{vu_{i}}\overrightarrow{e_{i}}\right) & =R\left(\frac{1}{\pi\left(v\right)}\sum_{i=1}^{n}\frac{\left\Vert \overrightarrow{f_{i}}\right\Vert }{\left\Vert \overrightarrow{e_{i}}\right\Vert }\cdot\overrightarrow{e_{i}}\right)=\frac{1}{\pi\left(v\right)}\sum_{i=1}^{n}\left\Vert \overrightarrow{f_{i}}\right\Vert R\left(\widehat{e_{i}}\right)=\frac{1}{\pi\left(v\right)}\sum_{i=1}^{n}\overrightarrow{f_{i}}.
\end{align*}
Since $\overrightarrow{f_{1}},\overrightarrow{f_{2}},...,\overrightarrow{f_{n}}$
trace a closed curve along the boundary of the polygon $P_{v}$, their
sum vanishes. Since $R$ is injective, the Theorem follows.
\end{proof}
\noindent We continue by analyzing some of the properties of the Dubejko
weights. Since the weights are defined using incircles, we will need
the following formula for their radius:
\begin{prop}
\label{prop: incircle radius formula}Let $x,y,z\in\mathbb{R}^{2}$
be centers of mutually tangent circles with respective radii $r_{x},r_{y},r_{z}$.
Then the inradius of the triangle $xyz$ is:
\[
r=\sqrt{\frac{r_{x}r_{y}r_{z}}{r_{x}+r_{y}+r_{z}}}.
\]
\end{prop}

\begin{proof}
We calculate the the area $A$ of the triangle $xyz$ in two ways.
On the one hand, the line segments connecting the incenter to $x,y,z$
divide the triangle into 3 triangles with altitudes $r$, and so:
\[
A=r\cdot\left(r_{x}+r_{y}+r_{z}\right).
\]
On the other hand, by Heron's formula we know:
\[
A=\sqrt{\left(r_{x}+r_{y}+r_{z}\right)r_{x}r_{y}r_{z}}.
\]
Equating the two above and solving for $r$ finishes the proof.
\end{proof}
The weights of the Dubejko-weighted random walk can be directly expressed
using the radii by the following formula, appearing in \cite{stephenson2005introduction}
after Theorem 18.3:
\begin{prop}
\label{Prop: Formula for Dubejko Weights}Let $\left\{ C_{v}\right\} _{v\in V}$
be a circle-packed infinite triangulation, and let $e=uv\in E$. Let
$w_{1},w_{2}\in V$ be the two vertices such that for every $i\in\left\{ 1,2\right\} $,
$w_{i}$ forms a face with $u$ and $v$. Then the weight $c_{e}$
is:
\[
c_{e}=\frac{\sqrt{r_{u}r_{v}}}{r_{u}+r_{v}}\cdot\left(\sqrt{\frac{R_{1}}{r_{u}+r_{v}+R_{1}}}+\sqrt{\frac{R_{2}}{r_{u}+r_{v}+R_{2}}}\right),
\]
where $r_{u},r_{v},R_{1},R_{2}$ are the radii of $C_{u},C_{v},C_{w_{1}},C_{w_{2}}$
respectively.
\end{prop}

\begin{proof}
For $i\in\left\{ 1,2\right\} $ set $r_{uvw_{i}}$ to be the inradius
of the triangle $uvw_{i}$. Using Proposition \ref{prop: incircle radius formula},
we can write:
\[
c_{e}=\frac{\left|e^{\dagger}\right|}{\left|e\right|}=\frac{r_{uvw_{1}}+r_{uvw_{2}}}{r_{u}+r_{v}}=\frac{\sqrt{r_{u}r_{v}}}{r_{u}+r_{v}}\cdot\left(\sqrt{\frac{R_{1}}{r_{u}+r_{v}+R_{1}}}+\sqrt{\frac{R_{2}}{r_{u}+r_{v}+R_{2}}}\right).
\]
\end{proof}
\begin{lem}
\label{cla:two radii are larger}Let $x,y,z\in\mathbb{R}^{2}$ be
centers of mutually tangent circles with radii $r_{x},r_{y},r_{z}$,
and let $r$ be the inradius of the triangle $xyz$. Then:
\[
r<\min\left\{ \sqrt{r_{x}r_{y}},\sqrt{r_{x}r_{z}},\sqrt{r_{y}r_{z}}\right\} .
\]
\end{lem}

\begin{proof}
By Proposition \ref{prop: incircle radius formula}, we can write:
\[
r=\sqrt{r_{x}r_{y}}\cdot\sqrt{\frac{r_{z}}{r_{x}+r_{y}+r_{z}}}<\sqrt{r_{x}r_{y}}.
\]
A similar calculation replacing the role of $z$ with $x,y$ shows
$r<\sqrt{r_{y}r_{z}},\ r<\sqrt{r_{x}r_{z}}$, finishing the proof.
\end{proof}
\begin{prop}
Let $\left\{ C_{v}\right\} _{v\in V}$ be a circle-packed infinite
triangulation, and $G=\left(V,E\right)$ its tangency graph. Then
for every edge $e\in E$ the weight is bounded from above by:\label{claim:weightsBoundedAbove}
\[
c_{e}<1.
\]
\end{prop}

\begin{proof}
Write $e=uv$ for $u,v\in V$, and let $w_{1},w_{2}\in V$ be the
two vertices forming a face with $u,v$, as in Figure \ref{fig: drawing for definition of Dubejko weights}.
Set $r_{u},r_{v}$ to be the radii of $C_{u},C_{v}$ and for each
$i\in\left\{ 1,2\right\} $ set $r_{uvw_{i}}$ to be the inradius
of the triangle $uvw_{i}$. The edge $e$ divides the dual edge $e^{\dagger}$
into two line segments of lengths $r_{uvw_{1}}$ and $r_{uvw_{2}}$.
Using Lemma \ref{cla:two radii are larger} and the AM-GM inequality
we bound:
\[
c_{e}=\frac{\left|e^{\dagger}\right|}{\left|e\right|}=\frac{r_{uvw_{1}}+r_{uvw_{2}}}{r_{u}+r_{v}}<\frac{2\sqrt{r_{u}r_{v}}}{r_{u}+r_{v}}\leq\frac{r_{u}+r_{v}}{r_{u}+r_{v}}=1.
\]
\end{proof}
A key result in the theory of circle packings is the famous Ring Lemma,
proven by Rodin and Sullivan in \cite{rodin1987convergence}:
\begin{lem}
\textbf{(Ring Lemma)}: \label{lem:(Ring-Lemma):-For}For each $d\in\mathbb{N}$,
there exists some $r=r\left(d\right)>0$ such that if a unit circle
is surrounded by $d$ circles forming a cycle externally tangent to
it, as in Figure \ref{fig: Ring Lemma Diagram}, then the radius of
each of the $d$ circles is larger than $r$.\\
\end{lem}

\begin{figure}[H]
\includegraphics[scale=0.45]{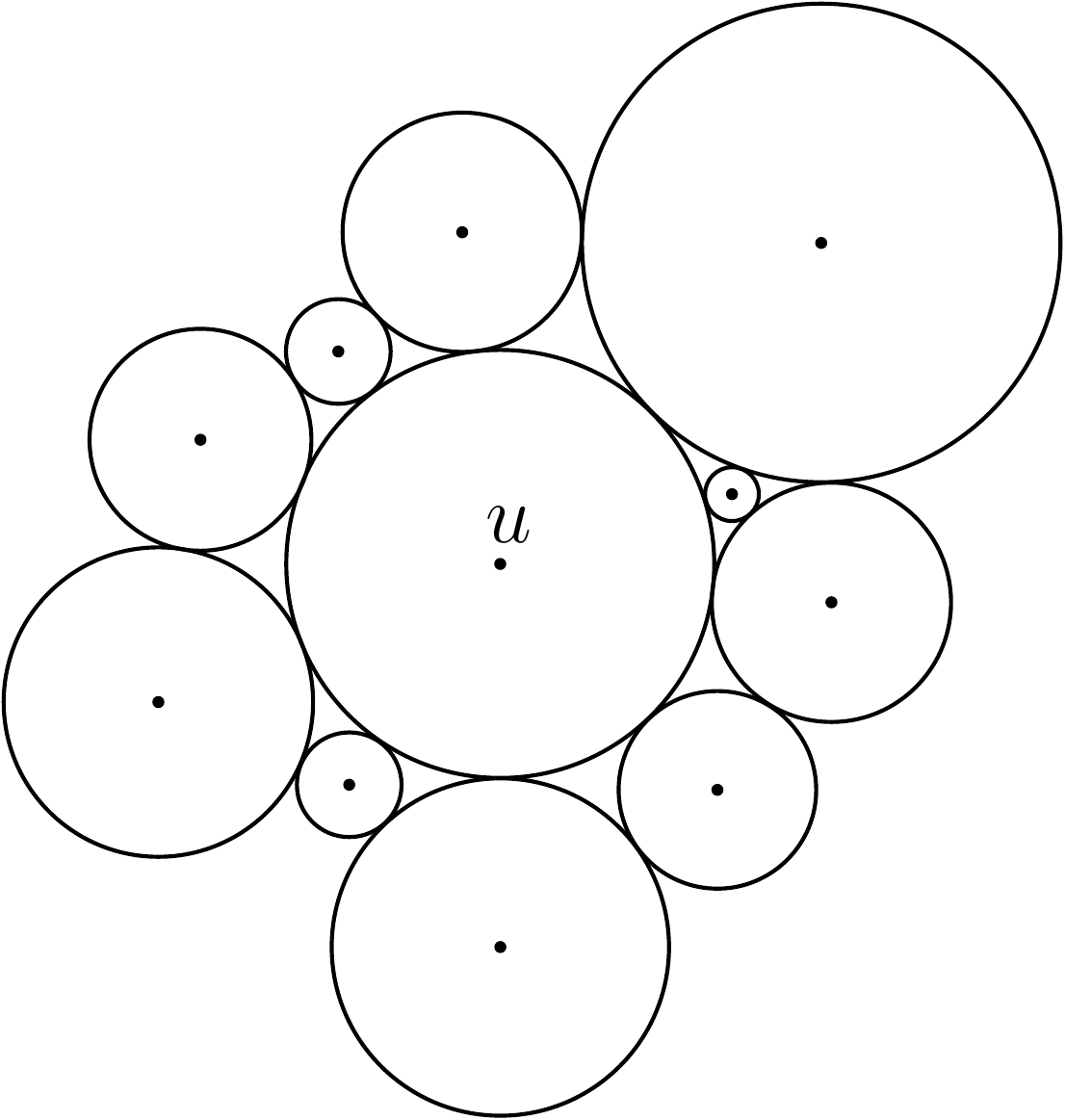}\caption{\label{fig: Ring Lemma Diagram}Circle $u$ surrounded by $d=10$
circles externally tangent to it}
\end{figure}
\noindent The following consequence of the Ring Lemma shows that under
the bounded degree assumption, since the radii of the circles around
a given circle cannot be too small, the corresponding edge weights
cannot be too small either:
\begin{prop}
\label{claim: weights bounded from below if bounded degree}Let $\left\{ C_{v}\right\} _{v\in V}$
be a circle-packed infinite triangulation and $G=\left(V,E\right)$
its tangency graph. If $G$ has degree bounded by $d$, then the weights
are bounded from below by some constant $c_{0}=c_{0}\left(d\right)>0$.
\end{prop}

\begin{proof}
Under the conditions of the Proposition, each circle in the packing
is surrounded by a cycle of length at most $d$ of externally tangent
circles. Thus, for each pair of adjacent circles, we can apply the
Ring Lemma \ref{lem:(Ring-Lemma):-For} in the wanted direction and
deduce that the ratio of their respective radii is smaller than some
global $M=M\left(d\right)>0$ (which can be taken to be the inverse
of $r\left(d\right)$ for example). \\
Let $e=uv\in E$ be an edge and let $w_{1},w_{2}$ the two vertices
such that $uvw_{1}$ and $uvw_{2}$ are faces of the circle packing's
drawing. Then using the formula from Proposition \ref{Prop: Formula for Dubejko Weights}
for the weights we bound:
\begin{align*}
c_{e} & =\sqrt{\frac{r_{u}}{r_{u}+r_{v}}}\cdot\sqrt{\frac{r_{v}}{r_{u}+r_{v}}}\cdot\left(\sqrt{\frac{R_{1}}{r_{u}+r_{v}+R_{1}}}+\sqrt{\frac{R_{2}}{r_{u}+r_{v}+R_{2}}}\right)\geq\\
 & \geq\sqrt{\frac{r_{u}}{r_{u}+Mr_{u}}}\cdot\sqrt{\frac{r_{v}}{Mr_{v}+r_{v}}}\cdot\left(\sqrt{\frac{R_{1}}{MR_{1}+MR_{1}+R_{1}}}+\sqrt{\frac{R_{2}}{MR_{2}+MR_{2}+R_{2}}}\right)=\\
 & =\frac{2}{M+1}\sqrt{\frac{1}{2M+1}},
\end{align*}
so we can take $c_{0}=\frac{2}{M+1}\sqrt{\frac{1}{2M+1}}$.
\end{proof}
\begin{prop}
Let $\left\{ C_{v}\right\} _{v\in V}$ be a circle-packed infinite
triangulation and $G=\left(V,E\right)$ its tangency graph, and let
$\left\{ c_{e}\right\} _{e\in E}$ be the Dubejko weights induced
by the circle packing. Then for any vertex $v\in V$ the sum of the
weights around $v$ is bounded by a constant: \label{claim: sum of weights bounded}
\[
\sum_{u\sim v}c_{vu}<2\pi.
\]
\end{prop}

\begin{proof}
The weights are invariant under dilation, so we can assume without
loss of generality that the radius of $C_{v}$ is $1$. Let $u_{1},u_{2}\in V$
be two vertices sharing a face with $v$, as drawn in Figure \ref{fig:Triangle vu1u2}:\\
\begin{figure}[H]
\includegraphics[scale=0.75]{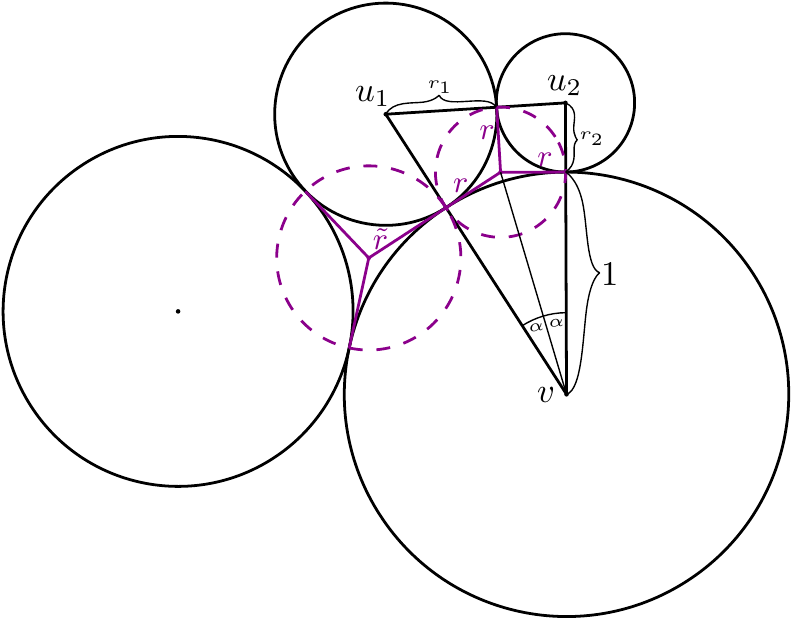}\caption{\label{fig:Triangle vu1u2}Triangle $vu_{1}u_{2}$ and its incircle}
\end{figure}
\noindent Denote by $r_{1},r_{2}$ their respective radii. Let $r$
be the inradius of the triangle $vu_{1}u_{2}$. Then by Proposition
\ref{prop: incircle radius formula} we have:
\[
r=\sqrt{\frac{r_{1}r_{2}}{1+r_{1}+r_{2}}}.
\]
Denote by $2\alpha$ the angle $\angle u_{1}vu_{2}$. The incenter
is the point of concurrency of the angle bisectors, so we have a right-angled
triangle with legs of length $1$ and $r$ and angle $\alpha$, so
$r=\tan\alpha$. Combining the two formulae for $r$ we get:
\[
\frac{1}{\cos^{2}\alpha}=1+\tan^{2}\alpha=1+r^{2}=\frac{1+r_{1}+r_{2}+r_{1}r_{2}}{1+r_{1}+r_{2}}=\frac{\left(1+r_{1}\right)\left(1+r_{2}\right)}{1+r_{1}+r_{2}}.
\]
We think of the weight $c_{vu_{1}}$ as ``split'' between two edges
connected in parallel with weights $\frac{r}{1+r_{1}}$ and $\frac{\tilde{r}}{1+r_{1}}$
where $\tilde{r}$ is the inradius of the triangle from the other
side of the edge $vu_{1}$. We can now bound the sum of the conductances
of the two half-edges relevant to the triangle $vu_{1}u_{2}$:
\[
\frac{r}{1+r_{1}}+\frac{r}{1+r_{2}}=r\cdot\frac{2+r_{1}+r_{2}}{\left(1+r_{1}\right)\left(1+r_{2}\right)}<r\cdot\frac{2+2r_{1}+2r_{2}}{\left(1+r_{1}\right)\left(1+r_{2}\right)}=\tan\alpha\cdot2\cos^{2}\alpha=\sin\left(2\alpha\right)<2\alpha.
\]
Summing the last inequality over all faces incident to $v$ gives
the desired result.
\end{proof}
\begin{rem}
The inequality of Proposition \ref{claim: sum of weights bounded}
is in fact tight: Consider a unit circle surrounded by $n$ identical
circles of radius $r_{n}$. Denote by $P_{n}$ the perimeter of an
$n$-regular polygon with inscribed radius $1$, then since $\lim_{n\rightarrow\infty}r_{n}=0$
and $\lim_{n\rightarrow\infty}P_{n}=2\pi$, the sum of weights around
the unit circle is:
\[
\frac{1}{1+r_{n}}\cdot P_{n}\underset{^{n\rightarrow\infty}}{\longrightarrow}2\pi.
\]
\end{rem}

\begin{lem}
\label{cor:Angle larger than arctan of sqrt of ratio}Let $x,y,z\in\mathbb{R}^{2}$
be centers of mutually tangent circles with radii $r_{x},r_{y},r_{z}$,
and let $r,M$ be the inradius and incenter of the triangle $xyz$.
Denote by $Q,R$ the respective tangency points of the incircle with
the sides $xy,xz$. Then:
\[
\measuredangle QMR>2\tan^{-1}\left(\sqrt{\frac{r_{x}}{\min\left\{ r_{y},r_{z}\right\} }}\right).
\]
\end{lem}

\noindent In the context of a circle packing, this Lemma tells us
that for a given face $xyz$, if at least one of $y,z$ has a comparable
radius to $x$ then the angle of the polygon $P_{x}$ at the incenter
of $xyz$ cannot be too small.
\begin{proof}
Set $\alpha=\measuredangle QMR$ and consider the right triangle $xMQ$:
The lengths of its altitudes are $\left|MQ\right|=r$ and $\left|xQ\right|=r_{x}$.
The line $xM$ bisects the angle $\measuredangle QMR$ and so:
\[
\frac{r_{x}}{r}=\tan\measuredangle QMx=\tan\frac{\alpha}{2}.
\]
Using Lemma \ref{cla:two radii are larger} we bound:
\[
\tan\frac{\alpha}{2}>\frac{r_{x}}{\sqrt{r_{x}r_{y}}}=\sqrt{\frac{r_{x}}{r_{y}}},
\]
and since $\tan^{-1}$ is strictly increasing we get $\alpha>2\tan^{-1}\left(\sqrt{\frac{r_{x}}{r_{y}}}\right)$.
A similar calculation replacing $y$ with $z$ shows $\alpha>2\tan^{-1}\left(\sqrt{\frac{r_{x}}{r_{z}}}\right)$,
and the result follows.
\end{proof}
\noindent The following beautiful theorem by Descartes (see §1.5 in
\cite{coxeter1969introduction}), illustrated in Figure \ref{fig:Descartes Theorem},
will also be of use to us:
\begin{thm}
\label{thm:(Descartes'-Theorem)}(Descartes' Theorem): Let $k_{1},k_{2},k_{3}$
be the curvatures (i.e. the reciprocal of the radius) of three circles
in the plane externally tangent to one another in 3 distinct points.
Then there are exactly two other circles (or one circle and one line)
tangent to all three, and their curvatures $k$ satisfy:
\[
k=k_{1}+k_{2}+k_{3}\pm2\sqrt{k_{1}k_{2}+k_{2}k_{3}+k_{3}k_{1}}
\]
where positive $k$, negative $k$ and $k=0$ represent a circle tangent
externally to all three, internally to all three and a straight line
respectively.
\end{thm}

\noindent
\begin{figure}[H]
\includegraphics[scale=0.6]{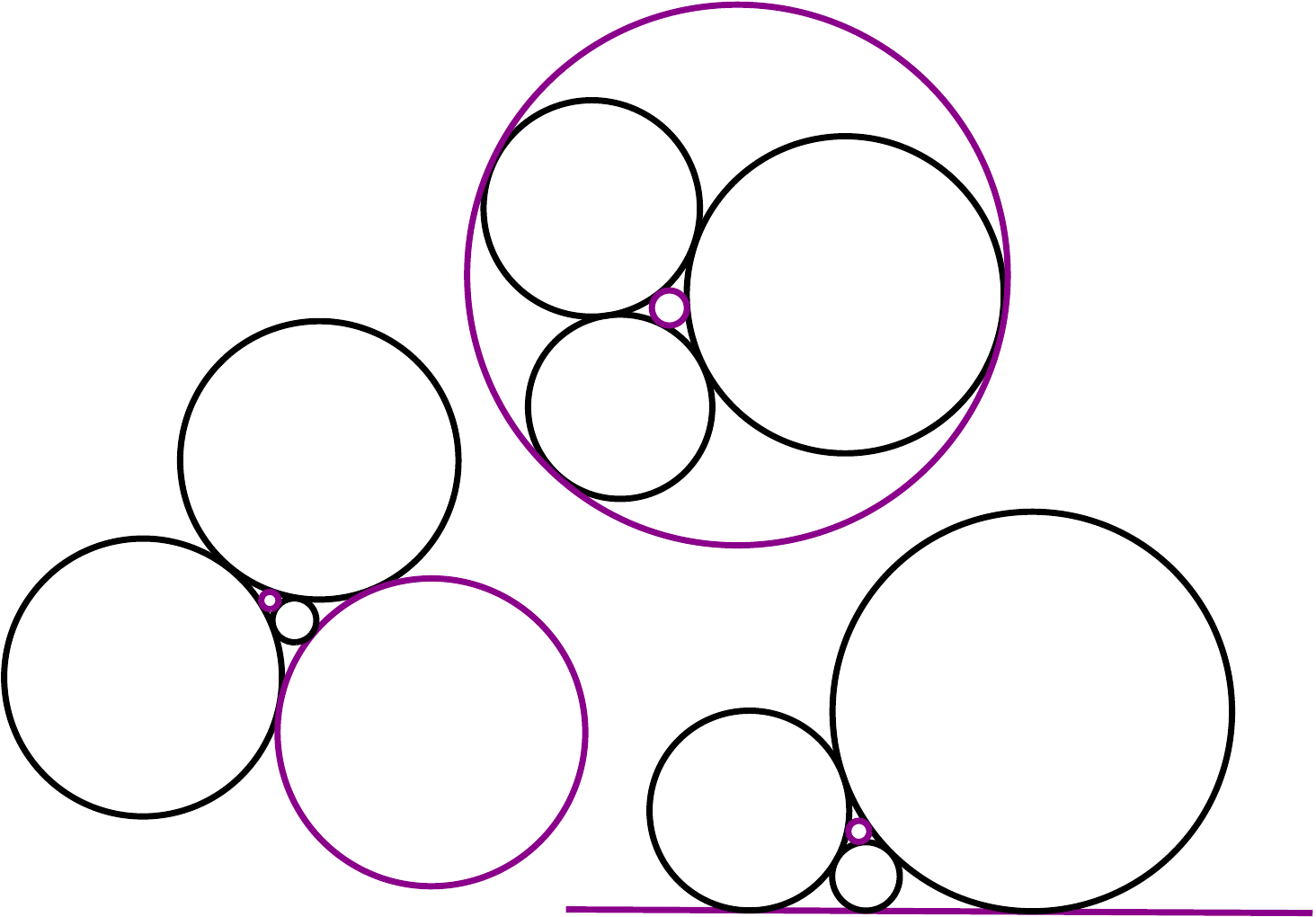}\caption{\label{fig:Descartes Theorem}3 mutually-tangent circles in black
and the 2 possibilities for a 4th circle in purple. A negative-curvature
solution can be seen in the topmost example and a zero-curvature solution
in the rightmost one.}
\end{figure}

\section{\label{sec:Special-Cases}Some Special Cases}

The goal of this section is to show two special cases exhibiting a
relation between the Dubejko-weighted random walk on an infinite circle-packed
triangulation and the parabolicity of its carrier. In Proposition
\ref{prop: bounded carrier implies transience}, whose arguments appear
in \cite{woess2000random} and \cite{stephenson2005introduction}
but are provided here for completeness, we show that if the carrier
is bounded then the WRW is transient. In Proposition \ref{prop: special case entire plane}
we show that if the carrier is the entire plane then the WRW is recurrent.
\begin{prop}
\label{prop: bounded carrier implies transience}Let $\Omega\subseteq\mathbb{R}^{2}$
be a bounded domain and $\left\{ C_{v}\right\} _{v\in V}$ a circle-packed
infinite triangulation of $\Omega$. Then the Dubejko-weighted random
walk on $\left\{ C_{v}\right\} _{v\in V}$ is transient.
\end{prop}

\begin{proof}
Let $\left(X_{n}\right)_{n\in\mathbb{N}}$ be the weighted random
walk started at some vertex $X_{0}\equiv\rho$. By Theorem \ref{Theorem: Dubejko Weights are a Martingale},
$\left(X_{n}\right)_{n\in\mathbb{N}}$ is a martingale. A bounded
martingale converges almost surely to some random variable (see for
example \cite{durrett2010probability}), and so $X_{n}\overset{a.s}{\longrightarrow}Y$
for some random variable $Y$. Choose any two vertices in the graph,
say $\rho$ and $\rho'$. Assume towards contradiction that the random
walk is recurrent, then $\rho$ and $\rho'$ are almost surely both
visited infinitely often. Now, $X_{n}$ converges to $Y$ a.s and
$X_{n}=\rho$ infinitely often a.s, which implies that $Y=\rho$ almost
surely. A similar argument shows that $Y=\rho'$ almost surely. But
$\rho\neq\rho'$ in contradiction.
\end{proof}
The proof of Proposition \ref{prop: bounded carrier implies transience}
was immediate from martingale arguments. For Proposition \ref{prop: special case entire plane},
we will use extensively the theory of probability and electric networks.
For background on this field one may read chapter 2 of \cite{lyons2017probability}.
\begin{defn}
Let $\left(G,c\right)$ be a network. The \textbf{Dirichlet energy}
of a function $f:V\left(G\right)\rightarrow\mathbb{R}$ is defined
by
\[
\mathcal{E}\left(f\right):=\sum_{uv\in E\left(G\right)}c_{uv}\left(f\left(v\right)-f\left(u\right)\right)^{2}.
\]
We will need the following criterion for recurrence which follows
from Dirichlet's Principle (see exercise 2.93 of \cite{lyons2017probability}):
\end{defn}

\begin{prop}
\label{prop: Criterion for recurrence}Let $\left(G,c\right)$ be
a network. Write $G=\left(V,E\right)$ and fix some $\rho\in V$.
Then $\left(G,c\right)$ is recurrent iff for any $\varepsilon>0$
there exists a finitely supported function $f:V\rightarrow\mathbb{R}$
with $f\left(\rho\right)=1$ and $\mathcal{E}\left(f\right)<\varepsilon$.
\end{prop}

\noindent In the case that the carrier is the entire plane, the following
Lemma shows that the effective resistance across any well-chosen annulus
in the plane is at least a constant, which will be shown later to
imply recurrence.
\begin{lem}
\label{lem: resistance accross any well-chosen annulus is atl least constant}Let
$\left\{ C_{v}\right\} _{v\in V}$ be a circle-packed infinite triangulation
of the entire plane. Then there exists some $C>0$ such that for any
$R>0$ there exists some finitely supported $f:V\rightarrow\mathbb{R}$
such that $f\mid_{V\cap B\left(0,R\right)}\equiv1$ and $\mathcal{E}\left(f\right)\leq C$.
\end{lem}

\begin{proof}
\noindent Assume WLOG that $R>0$ is large enough such that if $0\in C_{v}$
for some $v\in V$ then $C_{v}\subseteq B\left(0,R\right)$. \\
Define a continuous $\phi:\mathbb{R}^{2}\rightarrow\mathbb{R}$ in
polar coordinates by:
\[
\phi\left(r,\theta\right)=\begin{cases}
1, & r\leq R\\
\frac{2R-r}{R}, & R\leq r\leq2R\\
0, & r\geq2R
\end{cases}.
\]
$\phi$ induces a function $f$ on $V$ by assigning to a vertex $v\in V$
the value that $\phi$ takes on the center of $C_{v}$. It is immediate
that $f\mid_{V\cap B\left(0,R\right)}\equiv1$. In addition, $f$
is finitely supported since its support is contained in the compact
subset $\overline{B\left(0,2R\right)}$ of the carrier, which can
intersect only finitely many polygons of $\left\{ C_{v}\right\} _{v\in V}$
(rigorously proven later in Lemma \ref{lem: compact set intersects finite number of polygons})
and hence only finitely many vertices of $V$.\\
We move on to bound the Dirichlet energy of $f$. Let $G=\left(V,E\right)$
be the tangency graph of $\left\{ C_{v}\right\} _{v\in V}$. For each
edge $uv\in E$ there is an orthodiagonal quadrilateral $Q_{uv}$
whose diagonals are $e$ and $e^{\dagger}$, as shown in Figure \ref{fig: ODCA}.\\
\begin{figure}[H]
\includegraphics[scale=0.8]{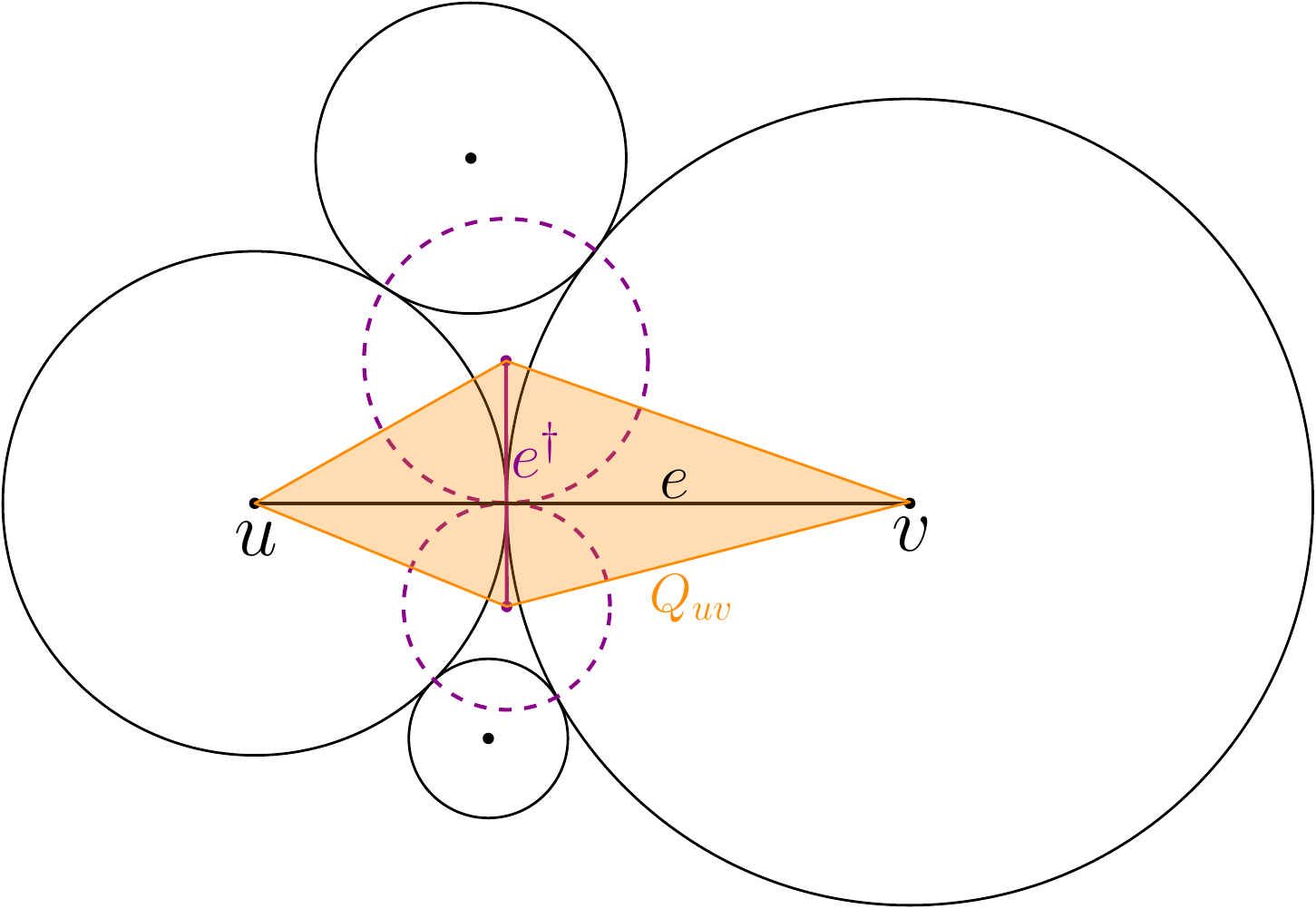}

\caption{\label{fig: ODCA}The orthodiagonal quadrilateral $Q_{uv}$.}
\end{figure}
 Using these quadrilaterals we divide the energy into two parts, bounding
each separately: Set $E_{1}=\left\{ e\in E\mid Q_{uv}\subseteq B\left(0,8R\right)\right\} $
and $E_{2}=E\backslash E_{1}$, and for each $i\in\left\{ 1,2\right\} $
set $\mathcal{E}_{i}=\sum_{uv\in E_{i}}\left(f\left(u\right)-f\left(v\right)\right)^{2}c_{uv}$.
Then $\mathcal{E}\left(f\right)=\mathcal{E}_{1}+\mathcal{E}_{2}$.\\
For the first part, notice that $\phi$ is $\frac{1}{R}$-Lipschitz
since the gradient's norm is:
\[
\left\Vert \nabla\phi\left(r,\theta\right)\right\Vert =\begin{cases}
\frac{1}{R}, & R<r<2R\\
0, & otherwise
\end{cases},
\]
and so for each $e=uv\in E$ we can bound:
\[
\left(f\left(u\right)-f\left(v\right)\right)^{2}c_{uv}\leq\left(\frac{1}{R}\left|e\right|\right)^{2}\frac{\left|e^{\dagger}\right|}{\left|e\right|}=\frac{2}{R^{2}}Area\left(Q_{e}\right).
\]
Since the quadrilaterals $\left\{ Q_{e}\right\} _{e\in E_{1}}$ have
disjoint interiors and are contained in $B\left(0,8R\right)$ we have:
\begin{align*}
\mathcal{E}_{1} & \leq\sum_{e\in E_{1}}\frac{2}{R^{2}}Area\left(Q_{e}\right)=\frac{2}{R^{2}}Area\left(\bigcup_{e\in E_{1}}Q_{e}\right)\leq\frac{2}{R^{2}}\cdot Area\left(B\left(0,8R\right)\right)=128\pi.
\end{align*}
We move on to bound the second term $\mathcal{E}_{2}$. Let $W=\left\{ v\in V:C_{v}\cap B\left(0,4R\right)\neq\emptyset\text{ and }r_{v}\geq4R\right\} $,
where $r_{v}$ is the radius of $C_{v}$. \\
We first claim that the cardinality of $W$ is bounded by some constant
number (say, 16). An easy way to show this is through area considerations:
If $v\in W$ then there exists some point $p\in C_{v}\cap B\left(0,4R\right)$.
This implies that there exists some $q\in B\left(0,6R\right)$ such
that $\left\Vert q-v\right\Vert \leq r_{v}-2R$: Indeed, set $t=\max\left\{ 0,1-\frac{r_{v}-2R}{\left\Vert p-v\right\Vert }\right\} $,
and then $q=p+t\left(v-p\right)$ can be shown to satisfy the above
conditions. By the choice of $q$ we have $B\left(q,2R\right)\subseteq C_{v}\cap B\left(0,8R\right)$,
and hence:
\[
\left|W\right|\cdot4R^{2}\leq\sum_{v\in W}Area\left(C_{v}\cap B\left(0,8R\right)\right)=Area\left(\bigcup_{v\in W}C_{v}\cap B\left(0,8R\right)\right)\leq Area\left(B\left(0,8R\right)\right)=64R^{2}.
\]
Dividing both sides by $4R^{2}$ shows that $\left|W\right|\leq16$.
\\
We next claim that if $e=uv\in E_{2}$ and $e$ has some energy contribution
then $u\in W$ or $v\in W$: If $f\left(u\right)=f\left(v\right)=0$
then the energy contribution is zero, so we can assume WLOG that $f\left(u\right)\neq0$,
and so $u\in B\left(0,2R\right).$ We will show that in this case
$v\in W$. First, we have $r_{u}<2R$, because if we assume otherwise
we have $0\in C_{u}$ and by the choice of $R$ we get $C_{u}\subseteq B\left(0,R\right)$
in contradiction. Thus, we have $C_{u}\subseteq B\left(0,4R\right)$
and since $C_{v}$ touches $C_{u}$ we deduce $C_{v}\cap B\left(0,4R\right)\neq\emptyset$.
Next, The quadrilateral $Q_{e}$ is convex, so from $Q_{e}\nsubseteq B\left(0,8R\right)$
we deduce that one of its vertices lies outside $B\left(0,8R\right)$.
This vertex cannot be $u$ since $u\in B\left(0,2R\right)$. If the
vertex is $v$, we have
\[
r_{v}=\left(r_{v}+r_{u}\right)-r_{u}=\left\Vert v-u\right\Vert -r_{u}\geq\left\Vert v\right\Vert -\left\Vert u\right\Vert -r_{u}>8R-2R-2R=4R,
\]
so $v\in W$ as needed. Otherwise, another vertex of the quadrilateral
which we denote by $z$ lies outside $B\left(0,8R\right)$, and let
$w\in V$ be the vertex such that $uvw$ is the face whose incenter
is $z$. Let $r$ be the inradius of $uvw$. Since $z\notin B\left(0,8R\right)$
and \uline{\mbox{$u\in B\left(0,2R\right)$}} we have:
\[
\left\Vert z-u\right\Vert \geq\left\Vert z\right\Vert -\left\Vert u\right\Vert >8R-2R=6R.
\]
On the other hand, using Lemma \ref{cla:two radii are larger} and
the orthogonality of $e$ and $e^{\dagger}$ we get:
\[
\left\Vert z-u\right\Vert =\sqrt{r_{u}^{2}+r^{2}}<\sqrt{r_{u}^{2}+r_{u}r_{v}}<\sqrt{4R^{2}+2R\cdot r_{v}}
\]
Combining the two inequalities involving $\left\Vert z-u\right\Vert $
and solving for $r_{v}$ gives $r_{v}>16R$, which finishes the second
claim about $W$.\\
Finally, using the trivial bound $0\leq f\leq1$ together Lemma \ref{claim: sum of weights bounded}
and what we know about $W$, we obtain:
\[
\mathcal{E}_{2}\leq\sum_{v\in W}\sum_{u\sim v}\left(f\left(u\right)-f\left(v\right)\right)^{2}c_{uv}\leq\sum_{v\in W}\sum_{u\sim v}1\cdot c_{uv}\leq\sum_{v\in W}2\pi\leq32\pi.
\]
In total we get $\mathcal{E}\left(f\right)=\mathcal{E}_{1}+\mathcal{E}_{2}\leq C$
for $C=160\pi$.
\end{proof}
\begin{prop}
\label{prop: special case entire plane}Let $\left\{ C_{v}\right\} _{v\in V}$
be a circle-packed infinite triangulation of the entire plane. Then
the Dubejko-weighted random walk on $\left\{ C_{v}\right\} _{v\in V}$
is recurrent.
\end{prop}

\begin{proof}
Let $G=\left(V,E\right)$ be the tangency graph of $\left\{ C_{v}\right\} _{v\in V}$
and fix some $\rho\in V$. Construct a sequence of radii $R_{1}<R_{2}<...$
and a sequence of finitely supported functions $f_{i}:V\rightarrow\mathbb{R}$
in the following way: Choose $R_{1}>0$ such that $\rho\in B\left(0,R_{1}\right)$,
and choose $f_{1}$ to be a finitely supported function as in Lemma
\ref{lem: resistance accross any well-chosen annulus is atl least constant}
such that $\mathcal{E}\left(f_{1}\right)<C$ and $f_{1}\mid_{V\cap B\left(0,R_{1}\right)}\equiv1$.
For each $i>1$, assume we've defined $R_{i}>0$ and $f_{i}$. Since
$f_{i}$ is finitely supported and $G$ is locally finite, we can
choose some $R_{i+1}>R_{i}$ such that for every $v\notin B\left(0,R_{i+1}\right)$
we have that $f_{i}$ vanishes on $v$ and on all its neighbours.
Applying Lemma \ref{lem: resistance accross any well-chosen annulus is atl least constant}
again, there exists some finitely supported $f_{i+1}:V\rightarrow\mathbb{R}$
with $\mathcal{E}\left(f_{i+1}\right)<C$ and $f_{i+1}\mid_{V\cap B\left(0,R_{i+1}\right)}\equiv1$.\\
For each $N$ set $g_{N}:=\sum_{i=1}^{N}\frac{1}{N}f_{i}$. Then $g_{N}$
is finitely supported (as the sum of finitely many finitely supported
functions). Furthermore, since $f_{i}\left(\rho\right)=1$ for every
$1\leq i\leq N$, we have $g_{N}\left(\rho\right)=1$ as their mean.
Using Proposition \ref{prop: Criterion for recurrence}, it is enough
to show that $\lim_{N\rightarrow\infty}\mathcal{E}\left(g_{N}\right)=0$
in order to deduce recurrence.\\
To this end, we claim that for every $uv\in E$ there is at most one
$1\leq i\leq N$ such that $f_{i}\left(u\right)\neq f_{i}\left(v\right)$.
Indeed, assume $f_{i}\left(u\right)\neq f_{i}\left(v\right)$. then
at least one of these is nonzero. By the construction of $f_{i}$,
this implies that $u,v\in B\left(0,R_{i+1}\right)$. Thus, for any
$j>i$ we have $f_{j}\mid_{V\cap B\left(0,R_{j}\right)}\equiv1$ and
in particular $f_{j}\left(u\right)=f_{j}\left(v\right)=1$. Similarly,
$f_{i}\left(u\right)$ and $f_{i}\left(v\right)$ cannot both be equal
to $1$, and WLOG we may assume that $f_{i}\left(u\right)\neq1$.
Then by construction of $f_{i}$ we have $u\notin B\left(0,R_{i}\right)$,
and so for each $j<i$ we have $u\notin B\left(0,R_{j+1}\right)$
and by the choice of $R_{j+1}$ we deduce $f_{j}\left(u\right)=f_{j}\left(v\right)=0$.\\
Now, having established that each edge $uv\in E$ contributes to the
energy of at most one $f_{i}$, the energy $\mathcal{E}\left(g_{N}\right)$
decomposes as follows:
\[
\mathcal{E}\left(g_{N}\right)=\mathcal{E}\left(\sum_{i=1}^{N}\frac{1}{N}f_{i}\right)=\sum_{i=1}^{N}\mathcal{E}\left(\frac{1}{N}f_{i}\right)=\sum_{i=1}^{N}\frac{1}{N^{2}}\mathcal{E}\left(f_{i}\right)<\sum_{i=1}^{N}\frac{1}{N^{2}}\cdot C=\frac{C}{N},
\]
and so $\lim_{N\rightarrow\infty}\mathcal{E}\left(g_{N}\right)=0$
as needed.
\end{proof}

\section{Some Integration Lemmas\label{sec: Some-Integration-Lemmas}}

The goal of this section is to prove Lemmas \ref{lem: integration on polygon}
and \ref{lem: Integration on parts of a polygon}, providing bounds
on integrals involving harmonic functions over polygons whose angles
are not too sharp.
\begin{defn}
Let $U\subseteq\mathbb{R}^{2}$ be open. A function $f:U\rightarrow\mathbb{R}$
is said to be \textbf{harmonic} in $U$ if it is twice continuously
differentiable and for every $x_{0}\in U$ the Laplacian vanishes
on $x_{0}$, i.e $\left(\frac{\partial^{2}f}{\partial x^{2}}+\frac{\partial^{2}f}{\partial y^{2}}\right)\left(x_{0}\right)=0$.
\end{defn}

\noindent We begin be recalling the following classical inequality
(for a proof see \cite{axler2013harmonic} for example):
\begin{thm}
\textbf{(Harnack's Inequality):} Let $f$ be nonnegative and harmonic
in $B\left(x_{0},R\right)$. Let $x\in B\left(x_{0},R\right)$ and
set $r=\left\Vert x-x_{0}\right\Vert $. Then:
\[
\frac{R-r}{R+r}f\left(x_{0}\right)\leq f\left(x\right)\leq\frac{R+r}{R-r}f\left(x_{0}\right).
\]
\end{thm}

\begin{lem}
Let $f$ be nonnegative and harmonic in $B\left(x_{0},R\right)$.
Then:\label{lem: nonnegative harmonic upper bound on gradient}
\[
\left\Vert \nabla f\left(x_{0}\right)\right\Vert \leq\frac{2f\left(x_{0}\right)}{R}.
\]

\end{lem}

\begin{proof}
Let $x\in B\left(x_{0},R\right)$, and write $r=\left\Vert x-x_{0}\right\Vert $.
By Harnack's inequality we have:
\[
\frac{R-r}{R+r}f\left(x_{0}\right)\leq f\left(x\right)\leq\frac{R+r}{R-r}f\left(x_{0}\right).
\]
Rearranging this we get:
\[
-\frac{2}{R+r}f\left(x_{0}\right)\leq\frac{f\left(x\right)-f\left(x_{0}\right)}{r}\leq\frac{2}{R-r}f\left(x_{0}\right).
\]
Using the inequality $R-r\leq R+r$, we deduce:
\[
\frac{\left|f\left(x\right)-f\left(x_{0}\right)\right|}{r}\leq\frac{2}{R-r}f\left(x_{0}\right).
\]
If $\nabla f\left(x_{0}\right)=0$ then the inequality holds. Otherwise,
Since $\left\Vert \nabla f\left(x_{0}\right)\right\Vert $ is the
directional derivative in the direction of $\hat{u}=\frac{\nabla f\left(x_{0}\right)}{\left\Vert \nabla f\left(x_{0}\right)\right\Vert }$,
we have:
\[
\left\Vert \nabla f\left(x_{0}\right)\right\Vert =\lim_{h\rightarrow0}\left|\frac{f\left(x_{0}+h\hat{u}\right)-f\left(x_{0}\right)}{h}\right|\leq\limsup_{h\rightarrow0}\frac{2}{R-h}f\left(x_{0}\right)=\frac{2}{R}f\left(x_{0}\right).
\]
\end{proof}
\begin{lem}
Let $f$ be harmonic in $B\left(x_{0},R\right)$ and continuous in
$\overline{B\left(x_{0},R\right)}$. The\label{lem: max-min bounds gradient}n:
\[
\left\Vert \nabla f\left(x_{0}\right)\right\Vert \leq\frac{2}{R}\left(\max_{x\in\overline{B\left(x_{0},R\right)}}f\left(x\right)-\min_{x\in\overline{B\left(x_{0},R\right)}}f\left(x\right)\right).
\]
\end{lem}

\begin{proof}
Set $g=f-\min_{x\in\overline{B\left(x_{0},R\right)}}f\left(x\right)$.
Then $g$ is nonnegative and harmonic in $B\left(x_{0},R\right)$,
and applying Lemma \ref{lem: nonnegative harmonic upper bound on gradient}
gives:
\[
\left\Vert \nabla f\left(x_{0}\right)\right\Vert =\left\Vert \nabla g\left(x_{0}\right)\right\Vert \leq\frac{2}{R}g\left(x_{0}\right)\leq\max_{x\in\overline{B\left(x_{0},R\right)}}f\left(x\right)-\min_{x\in\overline{B\left(x_{0},R\right)}}f\left(x\right).
\]
\end{proof}
\begin{lem}
There exists some $C_{0}>0$ such that for every $R>0$, $x_{0}\in\mathbb{R}^{2}$
and $f$ harmonic in $B\left(x_{0},R\right)$ and continuous in $\overline{B\left(x_{0},R\right)}$
we have in polar coordinates:\label{claim: integration on disc}
\[
\intop_{0}^{2\pi}\intop_{0}^{R}\left\Vert \nabla f\left(x_{0}+re^{i\theta}\right)\right\Vert ^{2}Rdrd\theta\leq C_{0}\intop_{0}^{2\pi}\intop_{0}^{R}\left\Vert \nabla f\left(x_{0}+re^{i\theta}\right)\right\Vert ^{2}rdrd\theta.
\]
\end{lem}

\begin{proof}
Notice first that both the LHS and the RHS integrals are invariant
under translations and dilations of the domain. Explicitly, defining
$\tilde{f}:\overline{B\left(0,1\right)}\rightarrow\mathbb{R}$ by
$\tilde{f}\left(x\right)=f\left(x_{0}+Rx\right)$ we have $\nabla\tilde{f}\left(x\right)=R\nabla f\left(x_{0}+Rx\right)$,
and so for the LHS:
\[
\intop_{0}^{2\pi}\intop_{0}^{1}\left\Vert \nabla\tilde{f}\left(re^{i\theta}\right)\right\Vert ^{2}drd\theta=\intop_{0}^{2\pi}\intop_{0}^{1}\left\Vert \nabla f\left(x_{0}+Rre^{i\theta}\right)\right\Vert ^{2}R^{2}drd\theta=\intop_{0}^{2\pi}\intop_{0}^{R}\left\Vert \nabla f\left(x_{0}+re^{i\theta}\right)\right\Vert ^{2}Rdrd\theta.
\]
Similarly, for the RHS:
\[
C_{0}\intop_{0}^{2\pi}\intop_{0}^{1}\left\Vert \nabla\tilde{f}\left(re^{i\theta}\right)\right\Vert ^{2}rdrd\theta=C_{0}\intop_{0}^{2\pi}\intop_{0}^{1}\left\Vert \nabla f\left(x_{0}+Rre^{i\theta}\right)\right\Vert ^{2}R^{2}rdrd\theta=C_{0}\intop_{0}^{2\pi}\intop_{0}^{R}\left\Vert \nabla f\left(x_{0}+re^{i\theta}\right)\right\Vert ^{2}rdrd\theta.
\]
Thus we may assume WLOG that $R=1$ and $x_{0}=0$. We present two
proofs for the claim: the first using Harnack's inequality and the
second using Fourier analysis.\\
For the first proof, set $\varepsilon_{d}$ to be the RHS (or the
Dirichlet energy), $\varepsilon_{d}=\intop_{0}^{2\pi}\intop_{0}^{1}\left\Vert \nabla f\left(re^{i\theta}\right)\right\Vert ^{2}rdrd\theta$,
and assume that for some $C_{0}>0$ we have:
\[
\intop_{0}^{2\pi}\intop_{0}^{1}\left\Vert \nabla f\left(re^{i\theta}\right)\right\Vert ^{2}drd\theta>C_{0}\cdot\varepsilon_{d}.
\]
Since $\intop_{0}^{2\pi}\intop_{\frac{1}{4}}^{1}\left\Vert \nabla f\left(re^{i\theta}\right)\right\Vert ^{2}drd\theta\leq\intop_{0}^{2\pi}\intop_{\frac{1}{4}}^{1}\left\Vert \nabla f\left(re^{i\theta}\right)\right\Vert ^{2}\cdot4r\cdot drd\theta\leq4\varepsilon_{d}$,
we have:
\begin{align*}
C_{0}\cdot\varepsilon_{d} & <\intop_{0}^{2\pi}\intop_{0}^{\frac{1}{4}}\left\Vert \nabla f\left(re^{i\theta}\right)\right\Vert ^{2}drd\theta+4\varepsilon_{d},
\end{align*}
so we find that:
\[
\intop_{0}^{2\pi}\intop_{0}^{\frac{1}{4}}\left\Vert \nabla f\left(re^{i\theta}\right)\right\Vert ^{2}drd\theta>\left(C_{0}-4\right)\varepsilon_{d}.
\]
Therefore, there exists some $x_{0}\in B\left(0,\frac{1}{4}\right)$
such that $\left\Vert \nabla f\left(x_{0}\right)\right\Vert ^{2}>\frac{4\left(C_{0}-4\right)}{2\pi}\varepsilon_{d}$.
By Lemma \ref{lem: max-min bounds gradient} applied to $f$ on $\overline{B\left(x_{0},\frac{1}{4}\right)}$
we deduce that:
\begin{equation}
\left(\max_{x\in\overline{B\left(x_{0},\frac{1}{4}\right)}}f\left(x\right)-\min_{x\in\overline{B\left(x_{0},\frac{1}{4}\right)}}f\left(x\right)\right)^{2}\geq\frac{\left(\frac{1}{4}\right)^{2}}{4}\cdot\left\Vert \nabla f\left(x_{0}\right)\right\Vert ^{2}>\frac{1}{64}\cdot\frac{4\left(C_{0}-4\right)}{2\pi}\varepsilon_{d}.\label{eq: max-min}
\end{equation}
Fix some $1>r\geq\frac{1}{2}$. By the maximum and minimum principles
for harmonic functions, $f$ restricted to $\overline{B\left(0,r\right)}$
achieves its maximum and minimum on the boundary $\partial B\left(0,r\right)$.
Thus, by inequality \ref{eq: max-min} and the fact that $\overline{B\left(x_{0},\frac{1}{4}\right)}\subseteq\overline{B\left(0,r\right)}$
we get that:
\[
\left(\max_{x\in\partial B\left(0,r\right)}f\left(x\right)-\min_{x\in\partial B\left(0,r\right)}f\left(x\right)\right)^{2}>\frac{1}{64}\cdot\frac{4\left(C_{0}-4\right)}{2\pi}\varepsilon_{d}.
\]
Let $\theta_{1},\theta_{2}$ be angles such that:
\[
f\left(re^{i\theta_{1}}\right)=\min_{x\in\partial B\left(0,r\right)}f\left(x\right),\ \ f\left(re^{i\theta_{2}}\right)=\max_{x\in\partial B\left(0,r\right)}f\left(x\right).
\]
Let $c$ be the shorter circular arc with center at $0$ going from
$re^{i\theta_{1}}$ to $re^{i\theta_{2}}$. We use Cauchy-Schwartz
for line integrals to get:
\begin{align*}
\left(f\left(re^{i\theta_{2}}\right)-f\left(re^{i\theta_{1}}\right)\right)^{2} & =\left(\intop_{c}\nabla f\cdot d\underline{r}\right)^{2}\leq\left(\intop_{c}\left\Vert \nabla f\right\Vert ^{2}ds\right)\cdot length\left(c\right)\leq\left(\intop_{\partial B\left(0,r\right)}\left\Vert \nabla f\right\Vert ^{2}ds\right)\cdot\pi r,
\end{align*}
And so:
\[
\intop_{0}^{2\pi}\left\Vert \nabla f\left(re^{i\theta}\right)\right\Vert ^{2}rd\theta>\frac{1}{\pi r}\cdot\left(\frac{1}{64}\cdot\frac{4\left(C_{0}-4\right)}{2\pi}\varepsilon_{d}\right)=\frac{\left(C_{0}-4\right)}{32\pi^{2}r}\cdot\varepsilon_{d}.
\]
Finally we get:
\begin{align*}
\varepsilon_{d} & =\intop_{0}^{\frac{1}{2}}\intop_{0}^{2\pi}\left\Vert \nabla f\left(re^{i\theta}\right)\right\Vert ^{2}rd\theta dr+\intop_{\frac{1}{2}}^{1}\intop_{0}^{2\pi}\left\Vert \nabla f\left(re^{i\theta}\right)\right\Vert ^{2}rd\theta dr\geq\\
 & \geq\intop_{\frac{1}{2}}^{1}\intop_{0}^{2\pi}\left\Vert \nabla f\left(re^{i\theta}\right)\right\Vert ^{2}rd\theta dr>\intop_{\frac{1}{2}}^{1}\frac{\left(C_{0}-4\right)}{32\pi^{2}r}\cdot\varepsilon_{d}dr=\frac{\left(C_{0}-4\right)\ln2}{32\pi^{2}}\cdot\varepsilon_{d}.
\end{align*}
This implies $C_{0}\leq4+\frac{32\pi^{2}}{\ln2}$, and the Lemma follows
for any value of $C_{0}$ larger than $4+\frac{32\pi^{2}}{\ln2}$.\\
For the second proof, assume that $f$ is smooth on the boundary.
Write the Fourier expansion of $f$ restricted to the unit circle:
$f\left(1,\theta\right)=a_{0}+\sum_{n=1}^{\infty}\left(a_{n}\cos\left(n\theta\right)+b_{n}\sin\left(n\theta\right)\right)$.
Then we have:
\[
f\left(re^{i\theta}\right)=a_{0}+\sum_{n=1}^{\infty}r^{n}\left(a_{n}\cos\left(n\theta\right)+b_{n}\sin\left(n\theta\right)\right).
\]
The gradient is:
\begin{align*}
\nabla f\left(re^{i\theta}\right) & =\frac{\partial f}{\partial r}\cdot\hat{r}+\frac{1}{r}\frac{\partial f}{\partial\theta}\cdot\hat{\theta}=\\
 & =\left(\sum_{n=1}^{\infty}nr^{n-1}\left(a_{n}\cos\left(n\theta\right)+b_{n}\sin\left(n\theta\right)\right)\right)\hat{r}+\left(\sum_{n=1}^{\infty}nr^{n-1}\left(-a_{n}\sin\left(n\theta\right)+b_{n}\cos\left(n\theta\right)\right)\right)\hat{\theta}.
\end{align*}
Using the orthogonality of $\hat{r},\hat{\theta}$ we find:
\begin{align*}
\left\Vert \nabla f\left(re^{i\theta}\right)\right\Vert ^{2} & =\left(\sum_{n=1}^{\infty}nr^{n-1}\left(a_{n}\cos\left(n\theta\right)+b_{n}\sin\left(n\theta\right)\right)\right)^{2}+\left(\sum_{n=1}^{\infty}nr^{n-1}\left(-a_{n}\sin\left(n\theta\right)+b_{n}\cos\left(n\theta\right)\right)\right)^{2}=\\
 & =\sum_{n,m=1}^{\infty}nmr^{n+m-2}\left(a_{n}\cos\left(n\theta\right)+b_{n}\sin\left(n\theta\right)\right)\left(a_{m}\cos\left(m\theta\right)+b_{m}\sin\left(m\theta\right)\right)+\\
 & +\sum_{n,m=1}^{\infty}nmr^{n+m-2}\left(-a_{n}\sin\left(n\theta\right)+b_{n}\cos\left(n\theta\right)\right)\left(-a_{m}\sin\left(m\theta\right)+b_{m}\cos\left(m\theta\right)\right).
\end{align*}
We integrate on theta first, and use the orthogonality of the Fourier
basis:
\begin{align*}
\intop_{0}^{2\pi}\left\Vert \nabla f\left(re^{i\theta}\right)\right\Vert ^{2}d\theta & =\sum_{n=1}^{\infty}n^{2}r^{2n-2}\left(a_{n}^{2}\cdot\pi+b_{n}^{2}\cdot\pi\right)+\sum_{n=1}^{\infty}n^{2}r^{2n-2}\left(a_{n}^{2}\cdot\pi+b_{n}^{2}\cdot\pi\right)=\\
 & =2\pi\sum_{n=1}^{\infty}n^{2}r^{2n-2}\left(a_{n}^{2}+b_{n}^{2}\right).
\end{align*}
Now the LHS of the Lemma's statement is:
\begin{align*}
\intop_{0}^{1}\intop_{0}^{2\pi}\left\Vert \nabla f\left(re^{i\theta}\right)\right\Vert ^{2}d\theta dr & =\intop_{0}^{1}\left(2\pi\sum_{n=1}^{\infty}n^{2}r^{2n-2}\left(a_{n}^{2}+b_{n}^{2}\right)\right)dr=\sum_{n=1}^{\infty}\frac{2\pi n^{2}}{2n-1}\left(a_{n}^{2}+b_{n}^{2}\right).
\end{align*}
The RHS is:
\[
\intop_{0}^{1}\intop_{0}^{2\pi}\left\Vert \nabla f\left(re^{i\theta}\right)\right\Vert ^{2}d\theta\cdot rdr=\intop_{0}^{1}\left(2\pi\sum_{n=1}^{\infty}n^{2}r^{2n-1}\left(a_{n}^{2}+b_{n}^{2}\right)\right)dr=\sum_{n=1}^{\infty}\frac{2\pi n^{2}}{2n}\left(a_{n}^{2}+b_{n}^{2}\right).
\]
For each $n\in\mathbb{N}$ we have $\frac{1}{2n-1}\leq\frac{2}{2n}$,
so taking $C_{0}=2$ finishes the proof.
\end{proof}
\begin{defn}
For any $\alpha>0$, a polygon $P$ is called \textbf{$\alpha$-nice}
if all of its angles are larger than $\alpha$.
\end{defn}

\begin{defn}
A polygon $P$ is called \textbf{tangential }if there is a circle
tangent to all of its sides.
\end{defn}

\begin{lem}
\label{lem: integration on polygon}For every angle $\alpha>0$ there
exists some $C_{1}\left(\alpha\right)>0$ such that for every $\alpha$-nice
tangential polygon $P$ and for every $f$ continuous in $P$ and
harmonic in the interior of $P$ we have:
\[
\intop_{\partial P}\intop_{T_{az}}\left\Vert \nabla f\left(w\right)\right\Vert ^{2}dwdz\leq C_{1}\left(\alpha\right)\cdot\iintop_{P}\left\Vert \nabla f\right\Vert ^{2}dA,
\]
where $a$ is the center of the circle inscribed in $P$ and $T_{az}$
is the straight line segment connecting $a$ to $z$.
\begin{proof}
\noindent Both sides of the inequality are invariant under translations
and dilations of $P$. Therefore, we may assume that $a=0$ and that
the radius of the inscribed circle is $1$. Let $\theta_{1}\leq\theta_{2}\leq...\leq\theta_{n}$
be the angles of the tangency points of the unit circle and the sides
of the polygon. Set $\theta_{n+1}=\theta_{1}$ for ease of notation.
For each $1\leq i\leq n$, write $\varphi_{i}=\frac{1}{2}\left(\theta_{i+1}-\theta_{i}\right)$
as in Figure \ref{fig: Tangential polygon}.\\
\begin{figure}[H]
\includegraphics[scale=0.6]{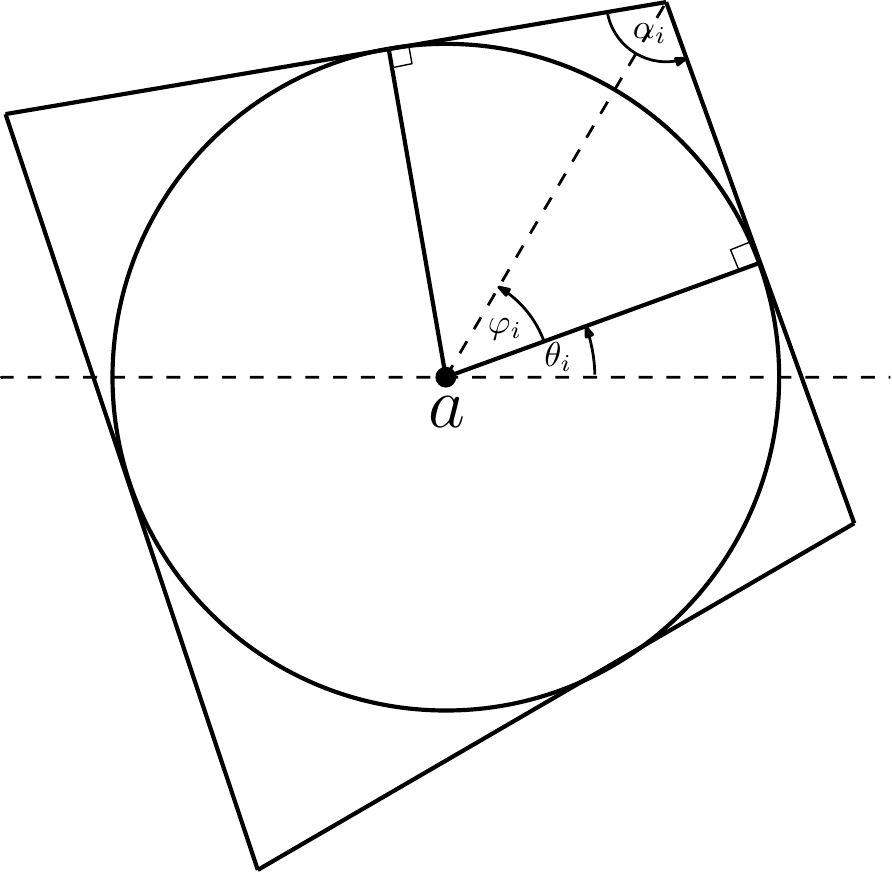}\caption{\label{fig: Tangential polygon}Tangential polygon and its incircle}
\end{figure}
 The boundary $\partial P$ of the polygon is given in polar coordinates
by:
\[
r\left(\theta\right)=\begin{cases}
\frac{1}{\cos\left(\theta-\theta_{i}\right)}, & 0\leq\theta-\theta_{i}\leq\varphi_{i}\\
\frac{1}{\cos\left(\theta_{i+1}-\theta\right)} & 0\leq\theta_{i+1}-\theta\leq\varphi_{i}
\end{cases}\ \ \ \Longrightarrow\ \ \ \frac{dr}{d\theta}=\begin{cases}
\frac{\sin\left(\theta-\theta_{i}\right)}{\cos^{2}\left(\theta-\theta_{i}\right)}, & 0\leq\theta-\theta_{i}\leq\varphi_{i}\\
-\frac{\sin\left(\theta_{i+1}-\theta\right)}{\cos^{2}\left(\theta_{i+1}-\theta\right)} & 0\leq\theta_{i+1}-\theta\leq\varphi_{i}
\end{cases}.
\]
 The length element is $ds=\sqrt{r^{2}+\left(\frac{dr}{d\theta}\right)^{2}}d\theta$,
and letting $\alpha_{i}=\pi-2\varphi_{i}$ be the $i$-th angle of
the polygon, we bound this factor:
\[
\sqrt{r^{2}+\left(\frac{dr}{d\theta}\right)^{2}}=\begin{cases}
\frac{1}{\cos^{2}\left(\theta-\theta_{i}\right)}, & 0\leq\theta-\theta_{i}\leq\varphi_{i}\\
\frac{1}{\cos^{2}\left(\theta_{i+1}-\theta\right)}, & 0\leq\theta_{i+1}-\theta\leq\varphi_{i}
\end{cases}\leq\frac{1}{\cos^{2}\left(\varphi_{i}\right)}=\frac{1}{\sin^{2}\left(\frac{\alpha_{i}}{2}\right)}\leq\frac{1}{\sin^{2}\left(\frac{\alpha}{2}\right)}.
\]
Setting $C_{1}'=\frac{1}{\sin^{2}\left(\frac{\alpha}{2}\right)}$,
we bound the integral:
\begin{align*}
\intop_{\partial P}\intop_{T_{0,z}}\left\Vert \nabla f\left(w\right)\right\Vert ^{2}dwdz & =\intop_{0}^{2\pi}\left(\intop_{0}^{r\left(\theta\right)}\left\Vert \nabla f\left(re^{i\theta}\right)\right\Vert ^{2}dr\right)\sqrt{r\left(\theta\right)^{2}+\left(\frac{dr\left(\theta\right)}{d\theta}\right)^{2}}d\theta\leq\\
 & \leq C_{1}'\cdot\intop_{0}^{2\pi}\intop_{0}^{r\left(\theta\right)}\left\Vert \nabla f\left(re^{i\theta}\right)\right\Vert ^{2}drd\theta.
\end{align*}
We now divide the polygon into two regions and bound the integral
on each of them: The first region is the unit disk $\mathbb{D}$,
on which we use Lemma \ref{claim: integration on disc} to get:
\[
\underbrace{\intop_{0}^{2\pi}\intop_{0}^{1}\left\Vert \nabla f\left(re^{i\theta}\right)\right\Vert ^{2}drd\theta}_{I_{1}}\leq C_{0}\iintop_{\mathbb{D}}\left\Vert \nabla f\right\Vert ^{2}dA\leq C_{0}\iintop_{P}\left\Vert \nabla f\right\Vert ^{2}dA.
\]
For the remaining region we have $r\geq1$ so we get:
\[
\underbrace{\intop_{0}^{2\pi}\intop_{1}^{r\left(\theta\right)}\left\Vert \nabla f\left(re^{i\theta}\right)\right\Vert ^{2}drd\theta}_{I_{2}}\leq\intop_{0}^{2\pi}\intop_{1}^{r\left(\theta\right)}\left\Vert \nabla f\left(re^{i\theta}\right)\right\Vert ^{2}\cdot rdrd\theta=\iintop_{P\backslash\mathbb{D}}\left\Vert \nabla f\right\Vert ^{2}dA\leq\iintop_{P}\left\Vert \nabla f\right\Vert ^{2}dA.
\]
Setting $C_{1}=C_{1}'\left(1+C_{0}\right)$ we finish the proof:
\begin{align*}
\intop_{\partial P}\intop_{T_{0,z}}\left\Vert \nabla f\left(w\right)\right\Vert ^{2}dwdz & \leq C_{1}'\cdot\intop_{0}^{2\pi}\intop_{0}^{r\left(\theta\right)}\left\Vert \nabla f\left(r\cdot e^{i\theta}\right)\right\Vert ^{2}drd\theta=C_{1}'\cdot\left(I_{1}+I_{2}\right)\leq\\
 & \leq C_{1}'\left(C_{0}+1\right)\iintop_{P}\left\Vert \nabla f\right\Vert ^{2}dA=C_{1}\iintop_{P}\left\Vert \nabla f\right\Vert ^{2}dA.
\end{align*}
\end{proof}
\end{lem}

\begin{defn}
Given a polygon $P$ and a side $S$ of $\partial P$, we say that
$S$ is an\textbf{ $\alpha$-nice side} of $P$ if both angles incident
to $S$ are larger than $\alpha$.
\end{defn}

\begin{lem}
\label{lem: Integration on parts of a polygon}Let $0<\alpha<\pi$.
Then in the notation of Lemma \ref{lem: integration on polygon} and
using the same constant $C_{1}\left(\alpha\right)$, for every $P$
tangential polygon, $S_{1},S_{2},...,S_{k}$ distinct $\alpha$-nice
sides of $P$ and $f$ continuous on $P$ and harmonic in the interior
of $P$ we have:
\[
\sum_{j=1}^{k}\intop_{S_{j}}\intop_{T_{az}}\left\Vert \nabla f\left(w\right)\right\Vert ^{2}dwdz\leq C_{1}\left(\alpha\right)\cdot\iintop_{P}\left\Vert \nabla f\right\Vert ^{2}dA.
\]
\end{lem}

\begin{proof}
\noindent We prove this by finding an $\alpha$-nice tangential polygon
$P'\subseteq P$. Consider a vertex $L$ of $\partial P$ with angle
$\beta$. If $\beta>\alpha$ we are done. Otherwise, we replace $L$
with two vertices $L_{1},L_{2}$ in the following way, illustrated
in Figure \ref{fig: Polygon vertex replacemt}. Denote by $X$ the
intersection of the line segment $aL$ and the incircle. Pass a tangent
to the circle at $X$, and denote by $L_{1},L_{2}$ the intersection
points of this line with the two tangents from $L$ to the circle.
For each $i\in\left\{ 1,2\right\} $, the angle facing the circle
at $L_{i}$ is $\frac{1}{2}\beta+\frac{1}{2}\pi$ since it is an exterior
angle to the triangle $LXL_{i}$ which has remote interior angles
$\frac{1}{2}\beta$ and $\frac{1}{2}\pi$. If $\frac{1}{2}\beta+\frac{\pi}{2}>\alpha$
we are done. Otherwise, perform this step again on each of $L_{1},L_{2}$
and repeat. After $n$ steps the angle is $\frac{1}{2^{n}}\beta+\left(1-\frac{1}{2^{n}}\right)\pi$
and since $\alpha<\pi$ and $\lim_{n\rightarrow\infty}\left(\frac{1}{2^{n}}\beta+\left(1-\frac{1}{2^{n}}\right)\pi\right)=\pi$,
for $n$ large enough we will have an angle larger than $\alpha$.
Perform this process on each vertex of $\partial P$ to get an $\alpha$-nice
tangential polygon $P'\subseteq P$ . Notice that all the tangency
points of $P$ are also tangency points in $P'$.
\begin{figure}[H]
\includegraphics[scale=0.4]{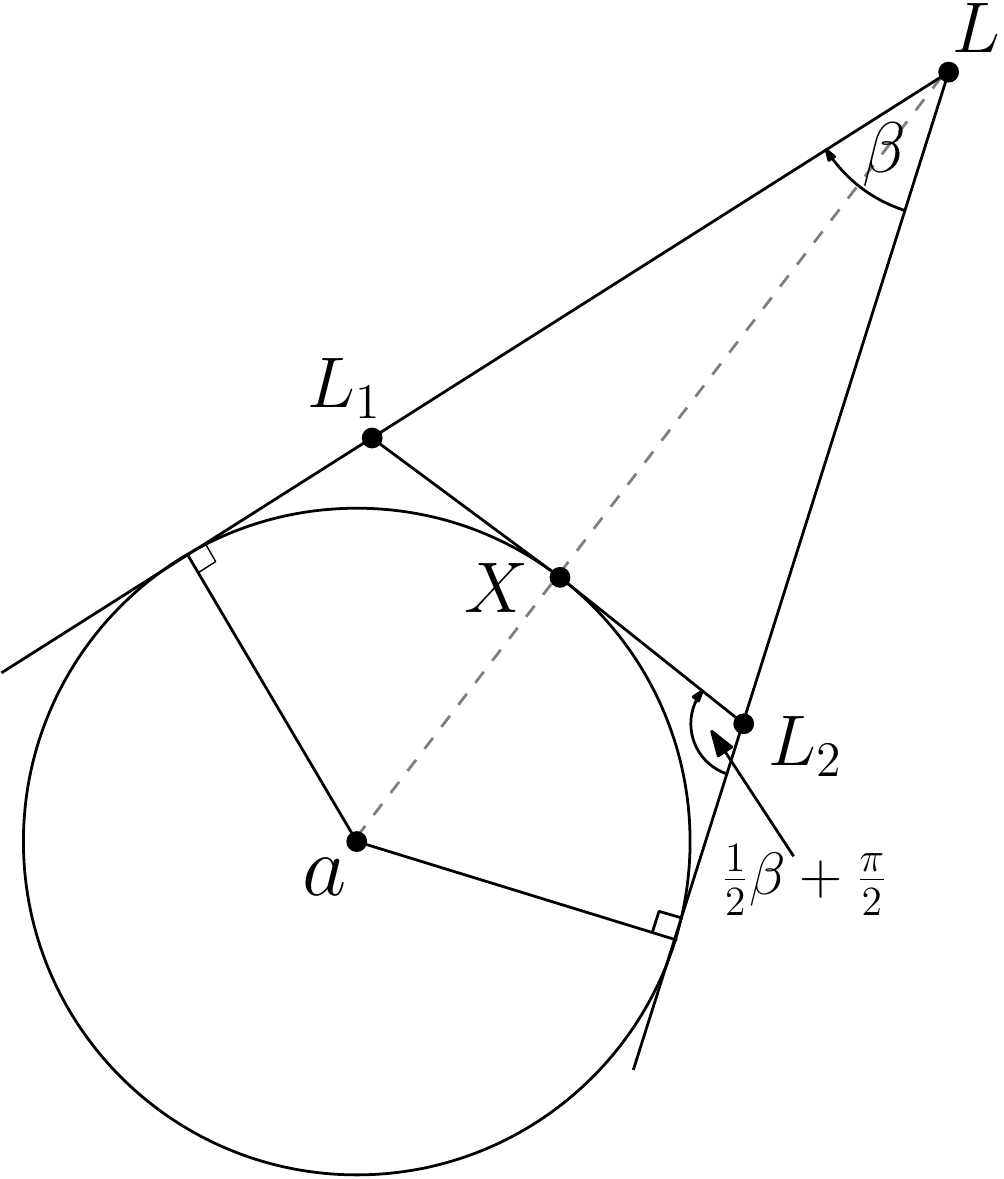}\caption{\label{fig: Polygon vertex replacemt}Polygon vertex replacement process.}
\end{figure}
For every $i$, the two vertices incident to the side $S_{i}$ had
angles larger than $\alpha$ to begin with, so $S_{i}$ is also a
side of $\partial P'$, and thus:
\[
\sum_{i=1}^{k}\intop_{S_{i}}\intop_{T_{az}}\left\Vert \nabla f\left(w\right)\right\Vert ^{2}dwdz\leq\intop_{\partial P'}\intop_{T_{az}}\left\Vert \nabla f\left(w\right)\right\Vert ^{2}dwdz.
\]
$f$ is continuous on $P'$ and harmonic in its interior because $P'\subseteq P$,
so using Lemma \ref{lem: integration on polygon} on $P'$ we get:
\[
\intop_{\partial P'}\intop_{T_{az}}\left\Vert \nabla f\left(w\right)\right\Vert ^{2}dwdz\leq C_{1}\left(\alpha\right)\cdot\iintop_{P'}\left\Vert \nabla f\right\Vert ^{2}dA.
\]
Finally, since $P'\subseteq P$:
\[
C_{1}\left(\alpha\right)\cdot\iintop_{P'}\left\Vert \nabla f\right\Vert ^{2}dA\leq C_{1}\left(\alpha\right)\cdot\iintop_{P}\left\Vert \nabla f\right\Vert ^{2}dA.
\]
 Putting together the three inequalities above we are done.
\end{proof}

\section{The General Case\label{sec:The-General-Case}}

\subsection{A Useful Coupling}

The goal of this subsection is to prove Corollary \ref{cor: coupling},
which says that given a circle-packed infinite triangulation, we can
add under certain conditions circles to the packing such that the
weighted random walks on the original and new circle packings can
be coupled. We begin with a few Lemmas about reversible Markov chains.
\begin{defn}
A Markov chain on a state space $V$ is said to be \textbf{represented
by a network} $\left(G,c\right)$ where $G$ has vertex set $V$ if
the chain and the weighted random walk on $\left(G,c\right)$ share
the same transition matrix.
\end{defn}

\noindent Note that in Definition \ref{def: network} of a network,
the set of edges $E$ is somewhat superfluous: It is sometimes more
convenient to compactly represent the network by the pair $\left(V,c\right)$
where $c:V\times V\rightarrow[0,\infty)$ is extended to take the
value 0 on pairs of vertices not in $E$, and the edge set $E$ can
be recovered as the support of $c$. \\
The weighted random walk on a network $\left(G,c\right)$ is an irreducible
reversible Markov chain: Irreducibility follows from the requirement
that $G$ be connected, and the sum $\pi\left(v\right)$ of the weights
around a vertex $v$ is easily seen to be a reversible measure. Conversely,
every irreducible reversible Markov chain can be represented by some
network: Indeed, given a reversible positive measure $\pi:V\rightarrow\left(0,\infty\right)$
on the states, a possible weight function is $c_{uv}=\pi\left(u\right)P_{u,v}$,
where $P$ is the chain's transition matrix. The weights of the representing
network are only unique up to a multiplicative constant. See chapter
2 of \cite{lyons2017probability} for some background on reversible
Markov chains.
\begin{defn}
Let $\left(X_{n}\right)_{n\in\mathbb{N}}$ be a Markov chain on a
state space $V$, and let $W\subseteq V$. The random sequence $\left(Y_{n}\right)_{n\in\mathbb{N}}$
obtained from $\left(X_{n}\right)_{n\in\mathbb{N}}$ by deleting all
appearances of $W$ is called the \textbf{Markov chain censored to
$V\backslash W$}.
\end{defn}

\noindent The subset $W$ usually has the property that the event
$\left\{ \exists N\text{ such that }\forall n\geq N,\ X_{n}\in W\right\} $
has probability zero so that the sequence $\left(Y_{n}\right)_{n\in\mathbb{N}}$
is almost surely an infinite sequence. In particular, this is the
case when the chain is irreducible and $W$ is a finite proper subset
of $V$. The censored chain (also known as the \emph{watched chain})
is itself a Markov chain, as shown in Lemma 6-6 of \cite{kemeny2012denumerable}.
The following Lemma shows that censoring also preserves reversibility,
and tells us how to transform the corresponding network when censoring
a single state:
\begin{lem}
\noindent \label{lem: censoring a single state}Let $\left(V,c\right)$
be a network and let $V':=V\backslash\left\{ w\right\} $ for some
$w\in V$. Then the censoring to $V'$ of the weighed random walk
on $\left(V,c\right)$ is a reversible Markov chain represented by
the network $\left(V',c'\right)$, where $c':V'\times V'\rightarrow[0,\infty)$
is given by $c'_{xy}=c_{xy}+\frac{c_{xw}c_{yw}}{\pi\left(w\right)-c_{ww}}$.
\end{lem}

\begin{proof}
\noindent We need to show that the weighted random walk on $\left(V',c'\right)$
satisfies the transition probabilities of the censored chain. Let
$x,y\in V'$. Denote by $\left(X_{n}\right)_{n\in\mathbb{N}}$ and
$\left(Y_{n}\right)_{n\in\mathbb{N}}$ the original and censored chains
respectively. The transition probabilities of the censored network
are:
\begin{align*}
\mathbb{P}_{x}\left(Y_{1}=y\right) & =\mathbb{P}_{x}\left(X_{1}=y\right)+\sum_{n=1}^{\infty}\mathbb{P}_{x}\left(X_{1}=X_{2}=...=X_{n}=w,X_{n+1}=y\right)=\\
 & =\frac{c_{xy}}{\pi\left(x\right)}+\sum_{n=1}^{\infty}\frac{c_{xw}}{\pi\left(x\right)}\cdot\left(\frac{c_{ww}}{\pi\left(w\right)}\right)^{n-1}\cdot\frac{c_{wy}}{\pi\left(w\right)}=\frac{1}{\pi\left(x\right)}\left(c_{xy}+\frac{c_{xw}c_{wy}}{\pi\left(w\right)}\frac{1}{1-\frac{c_{ww}}{\pi\left(w\right)}}\right)=\\
 & =\frac{c'_{xy}}{\pi\left(x\right)}.
\end{align*}
We next claim that the sum of weights $\pi'\left(x\right)$ around
a vertex $x$ in the new network $\left(V',c'\right)$ is the same
as in the old network. Indeed,
\begin{align*}
\pi'\left(x\right)=\sum_{y\in V'}c'_{xy} & =\left(\sum_{y\in V'}c_{xy}\right)+\frac{c_{xw}}{\pi\left(w\right)-c_{ww}}\sum_{y\in V'}c_{yw}=\left(\sum_{y\in V'}c_{xy}\right)+c_{xw}=\pi\left(x\right).
\end{align*}
Putting the equalities above together we get $\mathbb{P}_{x}\left(Y_{1}=y\right)=\frac{c'_{xy}}{\pi'\left(x\right)}$
as needed.
\end{proof}
\begin{rem}
Some network reduction techniques, such as the $Y-\Delta$ transform
and resistors in series, are special cases of Lemma \ref{lem: censoring a single state}.
\end{rem}

\begin{cor}
\label{cor: censoring a finite aount of states}Let $\left(V,c\right)$
be a network and let $V':=V\backslash W$ for some finite proper subset
$W\subseteq V$. Then the censoring to $V'$ of the weighed random
walk on $\left(V,c\right)$ is reversible and is represented by a
network $\left(V',c'\right)$ such that for every $x,y\in V'$ that
are not both in $\partial W=\left\{ v\in V'\mid\exists w\in W,\ c_{vw}>0\right\} $
we have $c'_{xy}=c_{xy}$.
\end{cor}

\begin{proof}
Write $W=\left\{ w_{1},w_{2},...,w_{n}\right\} $. The censored chain
to $V'$ is exactly the chain obtained by censoring the states in
$W$ one after the other. Hence, by applying Lemma \ref{lem: censoring a single state}
$n$ times, the chain censored to $V'$ is reversible. For the second
part, let $x,y\in V'$ with $x\notin\partial W$. Then for every $i\in\left\{ 1,2,...,n\right\} $
we have $c_{xw_{i}}=0$, and hence the weight $c_{xy}$ is unchanged
in each application of Lemma \ref{lem: censoring a single state}.
\end{proof}
\noindent When censoring a chain to $V\backslash W$, we may be adding
loops at vertices of $\partial W$. The following Lemma shows that
deleting repetitions from the chain corresponds to deleting loops
from the representing network:
\begin{defn}
Let $\left(X_{n}\right)_{n\in\mathbb{N}}$ be an irreducible Markov
chain. Its induced \textbf{repetition-deleted chain} is the random
sequence obtained from $\left(X_{n}\right)_{n\in\mathbb{N}}$ by replacing
consecutive appearances of the same state with a single appearance.
Formally, set $\tau_{0}=0$ and for every $i>0$ set $\tau_{i}=\inf\left\{ n>\tau_{i-1}\mid X_{n}\neq X_{\tau_{i-1}}\right\} $.
Then the repetition-deleted chain induced by $\left(X_{n}\right)_{n\in\mathbb{N}}$
is $\left(X_{\tau_{i}}\right)_{i\in\mathbb{\mathbb{N}}}$.
\end{defn}

\begin{lem}
\label{lem:deleting repetitions}Let $\left(X_{n}\right)_{n\in\mathbb{N}}$
be a Markov chain represented by a network $\left(V,c\right)$. Then
the repetition-deleted chain $\left(Y_{n}\right)_{n\in\mathbb{N}}$
induced by $\left(X_{n}\right)_{n\in\mathbb{N}}$ is reversible. Furthermore,
a network $\left(V,c'\right)$ representing $\left(Y_{n}\right)_{n\in\mathbb{N}}$
is obtained from $\left(V,c\right)$ by deleting self-loops.
\end{lem}

\begin{proof}
The new weights $c'$ are given by:
\[
c'_{xy}=\begin{cases}
c_{xy}, & x\neq y\\
0, & x=y
\end{cases},
\]
and so the new sum of weights around a vertex $x$ is:
\[
\pi'\left(x\right)=\sum_{y\in V}c'_{xy}=\left(\sum_{y\in V}c{}_{xy}\right)-c_{xx}=\pi\left(x\right)-c_{xx}.
\]
We need to show that the weighted random walk on $\left(V,c'\right)$
satisfies the transition probabilities of the repetition-deleted chain.
Indeed, if $y\neq x$ then:
\begin{align*}
\mathbb{P}_{x}\left(Y_{1}=y\right) & =\sum_{n=0}^{\infty}\mathbb{P}_{x}\left(X_{0}=X_{1}=...=X_{n}=x,X_{n+1}=y\right)=\\
 & =\sum_{n=0}^{\infty}\left(\frac{c_{xx}}{\pi\left(x\right)}\right)^{n}\cdot\frac{c_{xy}}{\pi\left(x\right)}=\frac{c_{xy}}{\pi\left(x\right)}\cdot\frac{1}{1-\frac{c_{xx}}{\pi\left(x\right)}}=\frac{c_{xy}}{\pi\left(x\right)-c_{xx}}=\frac{c'_{xy}}{\pi'\left(x\right)}.
\end{align*}
Otherwise, $y=x$ and then by the definition of $\left(Y_{n}\right)_{n\in\mathbb{N}}$
we have $\mathbb{P}_{x}\left(Y_{1}=x\right)=0=\frac{c'_{xx}}{\pi'\left(x\right)}$.
\end{proof}
The following Lemma shows that when removing the set of circles contained
in a triangle of an infinite circle-packed triangulation, if this
set is finite then the original weighted random walk and the weighted
random walk can be coupled:
\begin{lem}
\label{lem:effective-conductances}Let $\left\{ C_{v}\right\} _{v\in V}$
be a circle-packed infinite triangulation. Let $xyz$ be a triangle
in the graph such that the set $W\subseteq V$ of the vertices contained
in the interior of $xyz$ is finite, as in Figure \ref{fig: Triangle xyz and neighbours of x}.
Let $\left(X_{n}\right)_{n\in\mathbb{N}}$ be a Dubejko-weighted random
walk on $\left\{ C_{v}\right\} _{v\in V}$ and let $\left(Y_{n}\right)_{n\in\mathbb{N}}$
be the chain obtained from $\left(X_{n}\right)_{n\in\mathbb{N}}$
by first censoring it to $V'=V\backslash W$ and then deleting repetitions.
Then $\left(Y_{n}\right)_{n\in\mathbb{N}}$ is a Dubejko-weighted
random walk on $\left\{ C_{v}\right\} _{v\in V'}$.
\end{lem}

\noindent
\begin{figure}[H]
\includegraphics[scale=0.65]{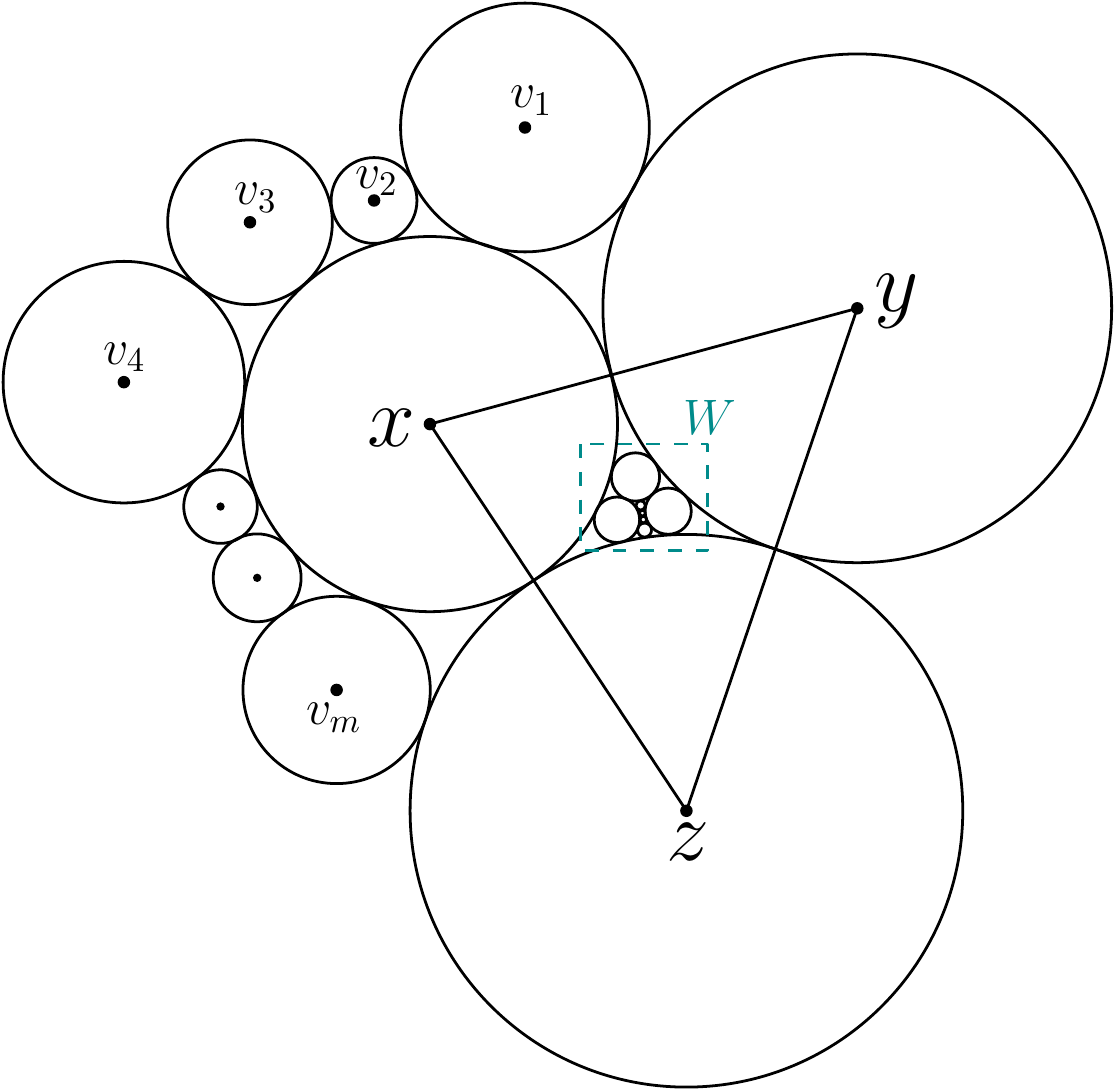}\caption{\label{fig: Triangle xyz and neighbours of x}Triangle $xyz$, the
set $W$ of vertices contained in the interior of $xyz$ and the neighbours
of $x$.}
\end{figure}
\begin{proof}
Let $G=\left(V,E\right)$ and $G'=\left(V',E'\right)$ be the respective
tangency graphs of $\left\{ C_{v}\right\} _{v\in V}$ and $\left\{ C_{v}\right\} _{v\in V'}$,
and let $c$ and $c'$ be their respective Dubejko weights. We first
notice that if $u_{1},u_{2}\in V'$ are not both in $\left\{ x,y,z\right\} $
then $c'_{u_{1},u_{2}}=c_{u_{1},u_{2}}$: Indeed, if $u_{1}u_{2}\notin E$
then $u_{1}u_{2}\notin E'$ and both sides of the equation are zero,
and otherwise equality holds as well since both the edge $u_{1}u_{2}$
and its dual have not changed in length when removing the circles
corresponding to vertices of $W$.\\
Next, note that $\partial W\subseteq\left\{ x,y,z\right\} $ (in fact
equality holds if $W$ is nonempty). Thus, by Corollary \ref{cor: censoring a finite aount of states}
and Lemma \ref{lem:deleting repetitions}, $\left(Y_{n}\right)_{n\in\mathbb{N}}$
is a reversible Markov chain represented by a network $\left(V',c''\right)$
with no self-loops such that for all $u_{1},u_{2}\in V'$ that are
not both in $\left\{ x,y,z\right\} $ we have $c''_{u_{1},u_{2}}=c_{u_{1},u_{2}}$.\\
We need to show that $c''\equiv c'$. For $u_{1},u_{2}\in V'$ that
are not both in $\left\{ x,y,z\right\} $ we've already seen that
$c'_{u_{1},u_{2}}=c_{u_{1},u_{2}}=c''_{u_{1},u_{2}}$. So we are left
with showing that $c''$ and $c'$ agree on edges of the triangle
$xyz$. Denote the neighbours of $x$ that are not in $W$ by $y,v_{1},v_{2},...,v_{m},z$,
as in Figure \ref{fig: Triangle xyz and neighbours of x}, and suppose
$X_{0}\equiv x$. For any vertex $u\in V$, let $\tau_{u}:=\min\left\{ n\geq0:X_{n}=u\right\} $.
Set a stopping time
\[
\tau=\min\left\{ \tau_{y},\tau_{v_{1}},\tau_{v_{2}},...,\tau_{v_{m}},\tau_{z}\right\} .
\]
$\tau$ is almost surely finite since $\left\{ \tau=\infty\right\} $
is the event that $\left(X_{n}\right)_{n\in\mathbb{N}}$ stays in
the finite set of vertices $W$ forever. For any $u\in\left\{ y,v_{1},v_{2},...,v_{m},z\right\} $,
the event $\left\{ \tau=\tau_{u}\right\} $ is exactly the event $\left\{ Y_{1}=u\right\} $.
By Theorem \ref{Theorem: Dubejko Weights are a Martingale}, the sequence
$\left(X_{n}\right)_{n\in\mathbb{N}}$ is a martingale. Stopping this
martingale at stopping time $\tau$ yields a bounded martingale, in
which case Doob's optional stopping theorem implies:
\[
x=\mathbb{P}_{x}\left(Y_{1}=y\right)\cdot y+\mathbb{P}_{x}\left(Y_{1}=z\right)\cdot z+\sum_{j=1}^{m}\mathbb{P}_{x}\left(Y_{1}=v_{j}\right)\cdot v_{j},
\]
or, passing $x$ to the other side:
\[
\mathbb{P}_{x}\left(Y_{1}=y\right)\cdot\left(y-x\right)+\mathbb{P}_{x}\left(Y_{1}=z\right)\cdot\left(z-x\right)+\sum_{j=1}^{m}\mathbb{P}_{x}\left(Y_{1}=v_{j}\right)\cdot\left(v_{j}-x\right)=0.
\]
Now, substituting $\mathbb{P}_{x}\left(Y_{1}=u\right)=\frac{c''_{xu}}{\pi''\left(x\right)}$
and multiplying the equation by $\pi''\left(x\right)$ we get:

\[
c''_{xy}\left(y-x\right)+c''_{xz}\left(z-x\right)+\sum_{j=1}^{m}c''_{xv_{j}}\left(v_{j}-x\right)=0.
\]
For every $j\in\left\{ 1,2,...,m\right\} $, since $v_{j}\notin\partial W$,
we have an equality $c''_{xv_{j}}=c_{xv_{j}}$, and thus:
\begin{equation}
c''_{xy}\left(y-x\right)+c''_{xz}\left(z-x\right)=-\sum_{j=1}^{m}c_{xv_{j}}\left(v_{j}-x\right).\label{eq: two tag conductances-1}
\end{equation}
Start a Dubejko-weighted random walk $\left(W_{n}\right)_{n\in\mathbb{N}}$
on $\left\{ C_{v}\right\} _{v\in V'}$ at $W_{0}\equiv x$. By Theorem
\ref{Theorem: Dubejko Weights are a Martingale} again, we have that
$\left(W_{n}\right)_{n\in\mathbb{N}}$ is a martingale, and in particular
$\mathbb{E}\left[W_{0}\right]=\mathbb{E}\left[W_{1}\right]$. A similar
calculation to the one we just did, substituting $\mathbb{P}_{x}\left(W_{1}=u\right)=\frac{c'_{xu}}{\pi'\left(x\right)},$
multiplying by $\pi'\left(x\right)$ and using $c'_{xv_{j}}=c_{xv_{j}}$,
gives us:
\begin{equation}
c'_{xy}\left(y-x\right)+c'_{xz}\left(z-x\right)=-\sum_{j=1}^{m}c_{xv_{j}}\left(v_{j}-x\right).\label{eq: tilde conductances-1}
\end{equation}
Combining equations \ref{eq: two tag conductances-1} and \ref{eq: tilde conductances-1}
yields:
\[
c''_{xy}\left(y-x\right)+c''_{xz}\left(z-x\right)=c'_{xy}\left(y-x\right)+c'_{xz}\left(z-x\right).
\]
$y-x$ and $z-x$ are two sides of a triangle and hence independent
vectors in $\mathbb{R}^{2}$. Thus, $c''_{xy}=c'_{xy}$ and $c''_{xz}=c'_{xz}$
as needed. A similar calculation with $y$ as center circle instead
of $x$ shows that $c''_{yz}=c'_{yz}$.
\end{proof}
The following Corollary is an extension of Lemma \ref{lem:effective-conductances}
which allows us, under certain conditions, to remove an infinite amount
of vertices from the packing:
\begin{cor}
\label{cor: coupling}Let $\left\{ C_{v}\right\} _{v\in V}$ be a
circle-packed infinite triangulation with tangency graph $G$. Let
$\left\{ x_{i}y_{i}z_{i}\right\} _{i\in I}$ for some $I=\left\{ 1,2,..,N\right\} $
or $I=\mathbb{N}$ be a family of triangles in $G$ with pairwise
disjoint interiors such that for each $i\in I$, the set of vertices
$W_{i}$ contained in the interior of the triangle $x_{i}y_{i}z_{i}$
is finite. Set $W=\bigcup_{i\in I}W_{i}$ and $V'=V\backslash W$.
For some $\rho\in V'$, start a Dubejko-weighted random walk $\left(X_{n}\right)_{n\in\mathbb{N}}$
on $\left\{ C_{v}\right\} _{v\in V}$ at $\rho$. Then the chain $\left(Y_{n}\right)_{n\in\mathbb{N}}$
obtained from $\left(X_{n}\right)_{n\in\mathbb{N}}$ by first censoring
it to $V'$ and then deleting repetitions is a Dubejko-weighted random
walk on $\left\{ C_{v}\right\} _{v\in V'}$ started at $\rho$. In
particular, this coupling implies that the weighted random walks on
$\left\{ C_{v}\right\} _{v\in V}$ and $\left\{ C_{v}\right\} _{v\in V'}$
are either both recurrent or both transient.
\end{cor}

\begin{proof}
We start by showing that the process $\left(Y_{n}\right)_{n\in\mathbb{N}}$
is well-defined in the sense that it is almost surely an infinite
sequence. \\
We first claim that for each $i\in I$ we have $x_{i},y_{i},z_{i}\in V'$:
Suppose otherwise, then WLOG $x_{i}\in W_{j}$ for some $j$. Since
$x_{i}\notin W_{i}$, we have $j\neq i$, and so $x_{i}\in int\left(conv\left\{ x_{j},y_{j},z_{j}\right\} \right)$.
This implies that some small ball around $x_{i}$ is contained in
$int\left\{ conv\left\{ x_{j},y_{j},z_{j}\right\} \right\} $. But
this ball also intersects the interior of $conv\left\{ x_{i},y_{i},z_{i}\right\} $,
in contradiction to the assumption of disjoint interiors for the triangles.
Thus, in the graph induced on $W$, each connected component is contained
in some $W_{i}$, and in particular each connected component is finite.
\\
We next claim that for every $m\in\mathbb{N}$, the event $A_{m}=\left\{ \exists t>m\text{ such that}\ X_{t}\in V'\text{ and }X_{t}\neq X_{m}\right\} $
has probability $1$. Since its complement is contained in the countable
union $\bigcup_{v\in V}B_{m,v}$ for $B_{m,v}=\left\{ X_{m}=v\wedge\left(\left(\forall t>m\right)X_{t}\in W\cup\left\{ v\right\} \right)\right\} $,
it is enough to show that $\mathbb{P}_{\rho}\left(B_{m,v}\right)=0$.
The connected component of $v$ in $W\cup\left\{ v\right\} $ is finite
because $v$ has finite degree and so its addition to $W$ can only
join together finitely many components of $W$, which are all finite.
Thus, the event $B_{m,v}$ implies that after time $m$, the random
walker never leaves the finite component of $v$ in $W\cup\left\{ v\right\} $,
so indeed $B_{m,v}$ has probability zero.\\
Finally, the event $A=\bigcup_{m\in\mathbb{N}}A_{m}$ has probability
$1$ since it is the countable union of such events. But $A$ implies
that the result $\left(Y_{n}\right)_{n\in\mathbb{N}}$ of censoring
$W$ and deleting repetitions is an infinite sequence since $A$ guarantees
that infinitely often the random walker has moved between different
states of $V'$. Thus, $\left(Y_{n}\right)_{n\in\mathbb{N}}$ is an
infinite sequence with probability $1$. \\
We move to show that $\left(Y_{n}\right)_{n\in\mathbb{N}}$ is a Dubejko-weighted
random walk on $\left\{ C_{v}\right\} _{v\in V'}$: Set $\left(Y_{n}^{0}\right)_{n\in\mathbb{N}}$
to be $\left(X_{n}\right)_{n\in\mathbb{N}}$, and for each $i\in I$
let $\left(Y_{n}^{i}\right)_{n\in\mathbb{N}}$ be obtained from $\left(Y_{n}^{i-1}\right)_{n\in\mathbb{N}}$
by censoring the vertices of $W_{i}$ and then deleting repetitions.
Fix some $k\in\mathbb{N}$. We claim that there exists some deterministic
$t_{k}\in\mathbb{N}$ such that $\left(Y_{0},Y_{1},...,Y_{k}\right)=\left(Y_{0}^{t_{k}},Y_{1}^{t_{k}},...,Y_{k}^{t_{k}}\right)$.
A key observation here is that for every $j$, the vertices in $\partial W_{j}=\left\{ x_{j},y_{j},z_{j}\right\} $
are neighbours to each other, and so $Y_{i+1}$ has to be a neighbour
of $Y_{i}$ in $G$ for every $i$. Hence, the first $k$ steps $Y_{0},Y_{1},...,Y_{k}$
are contained in a ball $B_{k}$ of graph-radius $k$ around $\rho$
in the graph induced on $V'$. Let $I_{k}=\left\{ i\in I\mid\left\{ x_{i},y_{i},z_{i}\right\} \cap B_{k}\neq\emptyset\right\} $.
Locally finiteness of $G$ implies that $I_{k}\subseteq\left\{ 1,2,...,t_{k}\right\} $
for some $t_{k}\in\mathbb{N}$, and so $\left(Y_{0},Y_{1},...,Y_{k}\right)=\left(Y_{0}^{t_{k}},Y_{1}^{t_{k}},...,Y_{k}^{t_{k}}\right)$.
Now, By Lemma \ref{lem:effective-conductances} applied $t_{k}$ times,
$\left(Y_{n}^{t_{k}}\right)_{n\in\mathbb{N}}$ is a weighted random
walk on $\left\{ C_{v}\right\} _{v\in V\backslash\bigcup_{i=1}^{t_{k}}W_{i}}$.
But the first $k$ steps in the weighted random walk on $\left\{ C_{v}\right\} _{v\in V\backslash\bigcup_{i=1}^{t_{k}}W_{i}}$
are equal in distribution to the first $k$ steps of the weighted
random walk on $\left\{ C_{v}\right\} _{v\in V'}$ since the balls
of graph-radius $k$ around $\rho$ in both networks are identical.
Thus, $\left(Y_{0},Y_{1},...,Y_{k}\right)$ also has the distribution
of the first $k$ steps of the weighted random walk on $\left\{ C_{v}\right\} _{v\in V'}$,
and since $k$ is arbitrary we get that $\left(Y_{n}\right)_{n\in\mathbb{N}}$
is a weighted random walk on $\left\{ C_{v}\right\} _{v\in V'}$ started
at $\rho$.\\
Finally, we show that the random walks are either both recurrent or
both transient. Assume first that $\left\{ C_{v}\right\} _{v\in V}$
is recurrent. Then $\rho$ is visited infinitely often in $\left(X_{n}\right)_{n\in\mathbb{N}}$
with probability $1$. The event that $\rho$ appears infinitely many
times in $\left(X_{n}\right)_{n\in\mathbb{N}}$ but only finitely
many times in $\left(Y_{n}\right)_{n\in\mathbb{N}}$ is contained
in the event $\bigcup_{m=1}^{\infty}B_{m,\rho}$ defined above, so
it has probability zero. Thus, $\rho$ almost surely appears infinitely
often in $\left(Y_{n}\right)_{n\in\mathbb{N}}$ and so $\left\{ C_{v}\right\} _{v\in V'}$
is recurrent. Now assume that $\left\{ C_{v}\right\} _{v\in V}$ is
transient. Then with positive probability $\rho$ appears finitely
often in $\left(X_{n}\right)_{n\in\mathbb{N}}$, and hence finitely
often in $\left(Y_{n}\right)_{n\in\mathbb{N}}$, implying that $\left\{ C_{v}\right\} _{v\in V'}$
is transient.
\end{proof}
\noindent The following Proposition is not used in the proof of Theorem
\ref{thm: main theorem-1}, but we've included it here because of
the proof's resemblance to that of Lemma \ref{lem:effective-conductances}:
\begin{prop}
Let $\left\{ C_{v}\right\} _{v\in V}$ be a circle-packed infinite
triangulation and $G=\left(V,E\right)$ its tangency graph. Let $xyz$
be a triangle in the graph such that the set $W$ of vertices contained
in the interior of the triangle $xyz$ is finite. Let $f:W\cup\left\{ x,y,z\right\} \rightarrow\mathbb{R}$
be harmonic in $W$ with respect to the Dubejko weights of the circle
packing. Then $f$ agrees with the unique affine function $g:\mathbb{R}^{2}\rightarrow\mathbb{R}$
taking the values $f\left(x\right),f\left(y\right),f\left(z\right)$
on $x,y,z$ respectively.
\end{prop}

\begin{proof}
Because $g$ is affine, its value on convex combinations of the form
$ax+by+cz$ for $a,b,c\geq0$ with $a+b+c=1$ is given by:
\[
g\left(ax+by+cz\right)=af\left(x\right)+bf\left(y\right)+cf\left(z\right).
\]
Let $w\in W\cup\left\{ x,y,z\right\} $. Start a weighted random walk
$\left(Z_{n}\right)_{n\in\mathbb{N}}$ at $Z_{0}\equiv w$, and stop
it at $\tau=\min\left\{ \tau_{x},\tau_{y},\tau_{z}\right\} $. By
Theorem \ref{Theorem: Dubejko Weights are a Martingale}, $\left(Z_{n\wedge\tau}\right)_{n\in\mathbb{N}}$
is a bounded martingale. By the optional stopping theorem we deduce
that $\mathbb{E}Z_{0}=\mathbb{E}Z_{\tau}$, so:
\[
w=\mathbb{P}_{w}\left(\tau=\tau_{x}\right)x+\mathbb{P}_{w}\left(\tau=\tau_{y}\right)y+\mathbb{P}_{w}\left(\tau=\tau_{z}\right)z.
\]
This presents $w$ (uniquely) as a convex combination of $\left\{ x,y,z\right\} $.
Since $f$ is harmonic, the sequence $\left(f\left(Z_{n\wedge\tau}\right)\right)_{n\in\mathbb{N}}$
is also a martingale. We apply the optional stopping theorem again
on this martingale, which is also bounded since $W$ is finite, to
get:
\begin{align*}
f\left(w\right) & =\mathbb{P}_{w}\left(\tau=\tau_{x}\right)f\left(x\right)+\mathbb{P}_{w}\left(\tau=\tau_{y}\right)f\left(y\right)+\mathbb{P}_{w}\left(\tau=\tau_{z}\right)f\left(z\right)=\\
 & =g\left(\mathbb{P}_{w}\left(\tau=\tau_{x}\right)x+\mathbb{P}_{w}\left(\tau=\tau_{y}\right)y+\mathbb{P}_{w}\left(\tau=\tau_{z}\right)z\right)=g\left(w\right).
\end{align*}
\end{proof}

\subsection{Parabolicity}
\begin{defn}
A domain $\Omega\subseteq\mathbb{R}^{2}$ is said to be \textbf{parabolic}
if for every open set $U\subseteq\Omega$ and for any $x\in\Omega$,
Brownian motion started at $x$ and killed at $\partial\Omega$ hits
$U$ almost surely.
\end{defn}

\begin{defn}
Let $\Omega\subseteq\mathbb{R}^{2}$ be a domain and $K\subseteq\Omega$
a compact set. The \textbf{capacity} of $K$ with respect to $\Omega$
is defined as:
\[
Cap\left(K,\Omega\right)=\inf\iintop_{\Omega}\left\Vert \nabla\phi\right\Vert ^{2}dA,
\]
where the infimum is taken over all Lipschitz functions $\phi:\Omega\rightarrow\mathbb{R}$
with compact support in $\Omega$ such that $\phi\mid_{K}\equiv1$
and $0\leq\phi\leq1$. For a precompact open set $K$ in $\Omega$
we define $Cap\left(K,\Omega\right)=Cap\left(\overline{K},\Omega\right)$.\\
The following equivalent condition for parabolicity, useful to us
due to its resemblence to Proposition \ref{prop: Criterion for recurrence},
is given in chapter $5$ of \cite{grigor1999analytic}:
\end{defn}

\begin{prop}
A domain $\Omega\subseteq\mathbb{R}^{2}$ is parabolic iff the capacity
of some/any precompact open set $K\subseteq\Omega$ is zero.\label{prop: Equivalent Condition for Parabolicity }
\end{prop}

\subsection{Proof of Theorem \ref{thm: main theorem-1}}
\begin{lem}
\label{lem: compact set intersects finite number of polygons}Let
$\left\{ C_{v}\right\} _{v\in V}$ be a circle-packed infinite triangulation
of a domain $\Omega\subseteq\mathbb{R}^{2}$, and let $A\subseteq\Omega$
be a compact subset. Then $A$ intersects finitely many polygons of
the packing.
\end{lem}

\begin{proof}
Every point $x\in\Omega$ has a neighbourhood $U_{x}\subseteq\Omega$
which intersects at most $3$ of the polygons: Since $\Omega=\bigcup_{v\in V}P_{v}$,
we know that $x\in P_{v}$ for some $v\in V$. If $x$ is in $int\left(P_{v}\right)$
then $x$ is contained in exactly one polygon. The tangency graph
of $\left\{ C_{v}\right\} _{v\in V}$ is a triangulation, and hence
the dual graph is $3$-regular. Thus, if $x$ is a vertex of $\partial P_{v}$
whose sides are edges of the dual graph, then $x$ lies in exactly
$3$ polygons. Lastly, if $x$ is on $\partial P_{v}$ but is not
a vertex then $x$ is in exactly $2$ polygons. In all the cases above,
a small enough ball around $x$ intersects at most $3$ polygons.
Since $A\subseteq\Omega=\bigcup_{x\in\Omega}U_{x}$, by compactness,
there exists a finite set $x_{1},x_{2},...,x_{m}\in\Omega$ such that
$A\subseteq\bigcup_{i=1}^{m}U_{x_{i}}$. Because $U_{x_{i}}$ intersects
at most $3$ polygons, $A$ intersects at most $3m$ polygons as needed.
\end{proof}
\begin{defn}
Let $\left\{ C_{v}\right\} _{v\in V}$ be a circle-packed infinite
triangulation with tangency graph $G=\left(V,E\right)$, and let $\alpha>0$.
An edge $e=uv\in E$ is said to be an \textbf{$\alpha$-nice edge}
of $\left\{ C_{v}\right\} _{v\in V}$ if its dual is an $\alpha$-nice
side of both $P_{u}$ and $P_{v}$.
\end{defn}

\begin{lem}
\label{lem: Graph replacement}Let $\left\{ C_{v}\right\} _{v\in V}$
be a circle-packed infinite triangulation of some domain $\Omega\subseteq\mathbb{R}^{2}$
and denote its tangency graph by $G=\left(V,E\right)$. Fix some vertex
$\rho\in V$. Then there exists some angle $\alpha_{0}>0$ such that
for every $\varepsilon>0$, by inserting new circles to the packing,
we can obtain a new circle-packed infinite triangulation $\left\{ C_{v}\right\} _{v\in V'}$
of $\Omega$ with tangency graph $G'=\left(V',E'\right)$ such that
for the Dubejko-weighted networks $\left(G,c\right)$ and $\left(G',c'\right)$
we have:
\begin{enumerate}
\item Set $E_{2}'=\left\{ e'\in E'\mid e'\text{ is not \ensuremath{\alpha_{0}}-nice}\right\} $.
Then:
\[
\sum_{e'\in E_{2}'}c'_{e'}\leq\varepsilon.
\]
\item ${\rm C_{eff}}\left(\rho\leftrightarrow\infty\right)={\rm C_{eff}}'\left(\rho\leftrightarrow\infty\right).$
In particular, $\left(G,c\right)$ is recurrent iff $\left(G',c'\right)$
is recurrent.
\end{enumerate}
\end{lem}

\begin{proof}
Choose $\alpha_{0}$ to be some positive angle such that $\alpha_{0}<10^{-5}$
and such that all of the polygons of $\rho$ and its neighbours are
$\alpha_{0}$-nice (this is possible by locally finiteness of $G$).\\
We insert circles in the following way: There is a countable amount
of faces in the triangulation (since they have nonempty disjoint interiors).
Enumerate them in some order $\left\{ F_{n}\right\} _{n\in\mathbb{N}}$
and fix some $n\in\mathbb{N}$. Write $F_{n}=\left\{ x,y,z\right\} $
for some $x,y,z\in V$ ordered by $r_{z}\geq r_{y}\geq r_{x}$, and
denote by $M$ and $r$ the incenter and inradius of the triangle
$xyz$. Denote the angles of $P_{x},P_{y},P_{z}$ at $M$ by $\alpha_{x},\alpha_{y},\alpha_{z}$.
Note that for each $v\in\left\{ x,y,z\right\} $ we have $\tan\frac{\alpha_{v}}{2}=\frac{r_{v}}{r}$
which implies $\tan\frac{\alpha_{z}}{2}\geq\tan\frac{\alpha_{y}}{2}\geq\tan\frac{\alpha_{x}}{2}$
and so $\alpha_{z}\geq\alpha_{y}\geq\alpha_{x}$. If $\alpha_{x}>\alpha_{0}$,
insert no new circles, since in this case $\alpha_{z},\alpha_{y},\alpha_{x}>\alpha_{0}$.
Otherwise, set $v_{0}=x$ and for each $i\in\left\{ 1,2,...,k\right\} $,
with $k=k\left(n\right)$ to be determined later, add a circle $C_{v_{i}}$
tangent to the three circles of $v_{i-1},y,z$ and contained in the
face $v_{i-1}yz$, effectively splitting this face into $3$ faces.
This process creates a chain of vertices $x=v_{0},v_{1},v_{2},...,v_{k}$
with respective radii $r_{x}=r_{0},r_{1},r_{2},...,r_{k}$ and with
centers approaching the tangency point of $C_{y}$ and $C_{z}$, as
in Figure \ref{fig: Circle insertion process-1}. Denote the new edge
weights by $c'$ and the new polygons by $P'$.
\begin{figure}[H]
\includegraphics[scale=0.8]{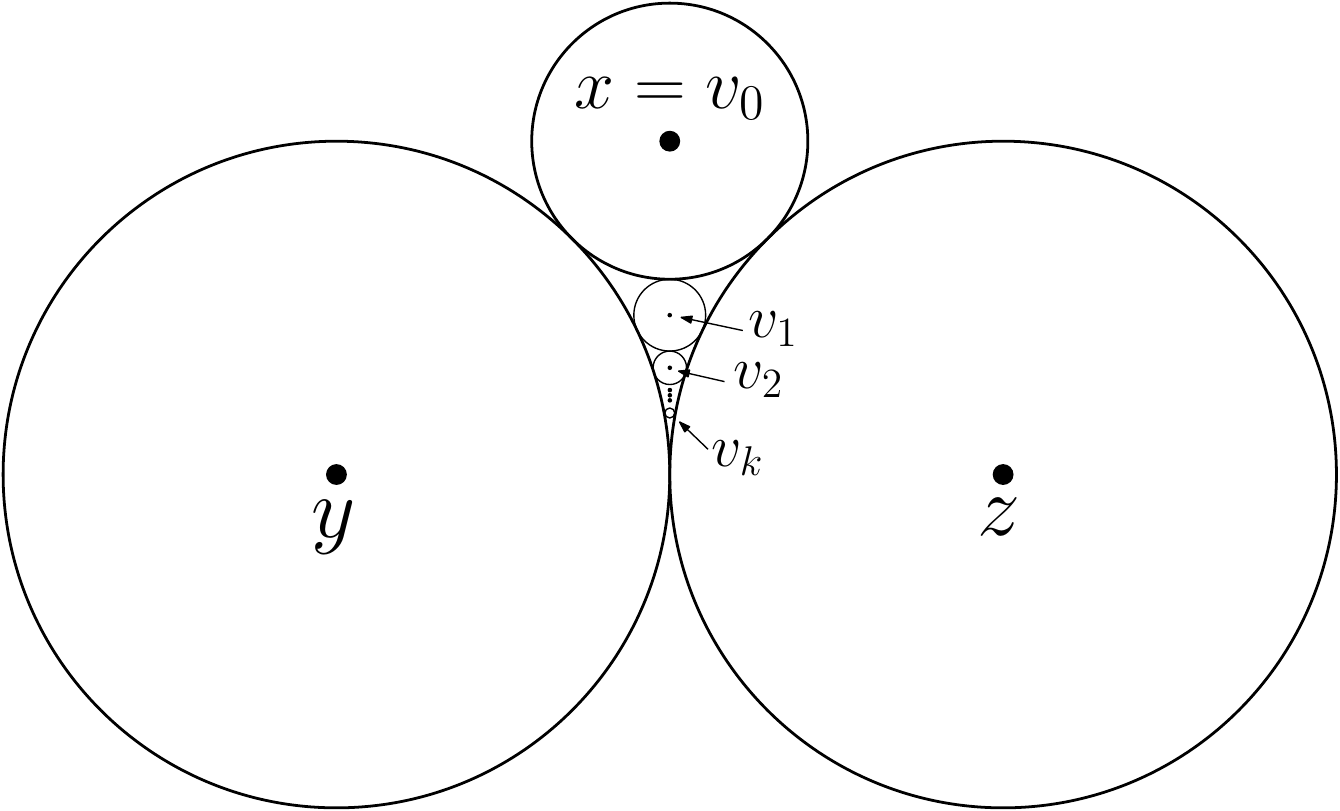}\caption{\label{fig: Circle insertion process-1}Circle insertion process}
\end{figure}
 We first show that $k=k\left(n\right)$ can be chosen such that $c'_{v_{k}y}+c'_{v_{k}z}<\frac{\varepsilon}{2^{n}}$:
Since the circles $C_{v_{1}},C_{v_{2}},...,C_{v_{k}}$ have disjoint
interiors and are contained in the triangle $xyz$, we have:
\[
\sum_{i=1}^{k}\pi r_{i}^{2}=\sum_{i=1}^{k}Area\left(C_{v_{i}}\right)=Area\left(\bigcup_{i=1}^{k}C_{v_{i}}\right)\leq Area\left(\triangle_{xyz}\right),
\]
so the series $\sum_{i=1}^{\infty}\pi r_{i}^{2}$ is bounded and in
particular the sequence of terms $\pi r_{k}^{2}$ tends to zero, so
$\lim_{k\rightarrow\infty}r_{k}=0$. Thus, by Proposition \ref{Prop: Formula for Dubejko Weights}:
\begin{align*}
c'_{v_{k}y} & =\frac{\sqrt{r_{k}r_{y}}}{r_{k}+r_{y}}\left(\sqrt{\frac{r_{z}}{r_{k}+r_{y}+r_{z}}}+\sqrt{\frac{r_{k-1}}{r_{k}+r_{y}+r_{k-1}}}\right)\leq\frac{\sqrt{r_{k}r_{y}}}{r_{k}+r_{y}}\cdot2\overset{_{k\rightarrow\infty}}{\longrightarrow}0.
\end{align*}
A similar calculation replacing $y$ with $z$ shows that $\lim_{k\rightarrow\infty}c'_{v_{k}z}=0$.\\
We next show that after the circle insertion process for the face
$F_{n}=\left\{ x,y,z\right\} $, out of all the incenters in the new
faces created, the only angle that might not be $\alpha_{0}$-nice
is the angle of $P'_{v_{k}}$ at the incenter of $v_{k}yz$: By Descartes'
Theorem \ref{thm:(Descartes'-Theorem)} for the four mutually-tangent
circles with centers $y,z,v_{i},v_{i+1}$ we get:
\[
\frac{1}{r_{i+1}}=\frac{1}{r_{y}}+\frac{1}{r_{z}}+\frac{1}{r_{i}}+2\sqrt{\frac{1}{r_{y}r_{z}}+\frac{1}{r_{y}r_{i}}+\frac{1}{r_{z}r_{i}}},
\]
where we have chosen the plus sign in the formula above because $r_{i+1}$
is the smaller-radius solution. In particular, $\frac{1}{r_{i+1}}\geq\frac{1}{r_{i}}$
so $r_{i+1}\leq r_{i}$ and by induction for all $i$ we have $r_{i}\leq r_{0}=r_{x}\leq r_{y,}r_{z}$
and so:
\[
\frac{r_{i}}{r_{i+1}}=1+\frac{r_{i}}{r_{y}}+\frac{r_{i}}{r_{z}}+2\sqrt{\frac{r_{i}^{2}}{r_{y}r_{z}}+\frac{r_{i}}{r_{y}}+\frac{r_{i}}{r_{z}}}\leq3+2\cdot\sqrt{3}\leq7.
\]
Denote by $Q_{i}$ the incenter of the triangle $yv_{i}v_{i+1}$ for
any $i\in\left\{ 0,1,...,k-1\right\} $. We claim that the respective
angles $\alpha,\beta,\gamma$ of $P'_{y},P'_{v_{i}},P'_{v_{i+1}}$
at $Q_{i}$ are larger than $\alpha_{0}:$ Indeed, since $\alpha_{0}<10^{-5}$
we can use Lemma \ref{cor:Angle larger than arctan of sqrt of ratio}
to deduce:
\begin{align*}
\alpha & >2\tan^{-1}\left(\sqrt{\frac{r_{y}}{r_{x}}}\right)\geq2\tan^{-1}\left(1\right)=\frac{\pi}{2}>\alpha_{0}.\\
\beta & >2\tan^{-1}\left(\sqrt{\frac{r_{i}}{r_{i+1}}}\right)\geq2\tan^{-1}\left(1\right)=\frac{\pi}{2}>\alpha_{0}.\\
\gamma & >2\tan^{-1}\left(\sqrt{\frac{r_{i+1}}{r_{i}}}\right)\geq2\tan^{-1}\left(\sqrt{\frac{1}{7}}\right)>\alpha_{0}.
\end{align*}
A similar calculation with $z$ replacing $y$ shows that for every
$i\in\left\{ 0,1,...,k-1\right\} $ the angles of $P_{z}',P_{v_{i}}',P_{v_{i+1}}'$
at the incenter of $zv_{i}v_{i+1}$ are larger than $\alpha_{0}.$
Denote by $T$ the incenter of $v_{k}yz$. Denote by $\delta_{y},\delta_{z}$
the respective angles of $P'_{y}$ and $P'_{z}$ at  $T$. Then by
Lemma \ref{cor:Angle larger than arctan of sqrt of ratio} again we
get for each $u\in\left\{ y,z\right\} $:
\begin{align*}
\delta_{u} & >2\tan^{-1}\left(\sqrt{\frac{r_{u}}{r_{k}}}\right)\geq2\tan^{-1}\left(1\right)=\frac{\pi}{2}>\alpha_{0}.
\end{align*}
So indeed, out of the angles of new polygons contained in $xyz$,
only the angle of $P'_{v_{k}}$ at the incenter of $v_{k}yz$ might
not be $\alpha_{0}$-nice. \\
For the first statement of the lemma, write $F_{n}=\left\{ x_{n},y_{n},z_{n}\right\} $
as above for every $n\in\mathbb{N}$. Set $I\subseteq\mathbb{N}$
to be the set of indices of faces where we've inserted new circles.
For every $i\in I$, let $s_{i}\in V'$ be the vertex of the last
circle inserted to $F_{n}$. We claim that $E_{2}'\subseteq\bigcup_{i\in I}\left\{ s_{i}y_{i},s_{i}z_{i}\right\} $:
Let $e'\in E_{2}'$. Then its dual is a side of a polygon such that
one of its endpoints has an angle that is smaller than $\alpha_{0}$.
By the angle analysis above, the only such angles are of $P'_{s_{i}}$
at the incenter of $s_{i}y_{i}z_{i}$ for some $i\in I$. So $e'$
must be either $s_{i}y_{i}$ or $s_{i}z_{i}$ as needed. Therefore,
we can bound:
\[
\sum_{e'\in E_{2}'}c'_{e'}\leq\sum_{i\in I}\left(c'_{s_{i}y_{i}}+c'_{s_{i}z_{i}}\right)\leq\sum_{i\in I}\frac{\varepsilon}{2^{i}}\leq\sum_{i\in\mathbb{N}}\frac{\varepsilon}{2^{i}}=\varepsilon.
\]
For the second statement, recall that (for a reference read for example
the paragraph preceding Theorem $2.3$ of \cite{lyons2017probability}):
\[
{\rm C_{eff}}'\left(\rho\leftrightarrow\infty\right)=\pi'\left(\rho\right)\cdot\mathbb{P}_{\rho}'\left[\rho\rightarrow\infty\right],
\]
where $\pi'\left(\rho\right)=\sum_{\substack{v\in V'\\
v\sim\rho
}
}c'_{\rho v}$, $\mathbb{P}_{\rho}'$ is the probability measure of the weighted
random walk on $\left(G',c'\right)$ started at $\rho$ and $\rho\rightarrow\infty$
is the event that the random walk never returns to $\rho$. In the
circle insertion process we have not added circles in the faces incident
to $\rho$ due to the choice of $\alpha_{0}$, and so $\pi\left(\rho\right)=\pi'\left(\rho\right)$.
In addition, Corollary \ref{cor: coupling} shows that the weighted
random walks on $\left\{ C_{v}\right\} _{v\in V'}$ started at $\rho$
can be coupled to the weighted random walk on $\left\{ C_{v}\right\} _{v\in V}$
by first censoring it to $V$ and then deleting repetitions. In the
circle insertion process we have not added vertices adjacent to $\rho$,
and so no appearances of $\rho$ will be deleted when censoring and
deleting repetitions. Thus, in the coupling we have that the random
walker on $G'$ returns to $\rho$ iff the random walker on $G$ returns
to $\rho$, and so $\mathbb{P}_{\rho}\left[\rho\rightarrow\infty\right]=\mathbb{P_{\rho}}'\left[\rho\rightarrow\infty\right]$.
We conclude:
\[
{\rm C_{eff}}\left(\rho\leftrightarrow\infty\right)=\pi\left(\rho\right)\mathbb{P}_{\rho}\left[\rho\rightarrow\infty\right]=\pi'\left(\rho\right)\mathbb{P}_{\rho}'\left[\rho\rightarrow\infty\right]={\rm C_{eff}}'\left(\rho\leftrightarrow\infty\right).
\]
Finally, the carrier of $\left\{ C_{v}\right\} _{v\in V'}$ is the
same as the carrier of $\left\{ C_{v}\right\} _{v\in V}$ since in
each face of the original packing we have added finitely many circles.
\end{proof}
\noindent We are now ready to prove the main result of this paper,
restated here:

\mainthm*
\begin{proof}
Fix a vertex $\rho\in V$ and let $\varepsilon>0$. We will show that
${\rm C_{eff}}\left(\rho\leftrightarrow\infty\right)=0$, which implies
recurrence.\\
\uline{Step 1: Insert new circles to the packing:} By Lemma \ref{lem: Graph replacement},
we may assume without loss of generality that there exists some angle
$\alpha_{0}>0$ such that for $E_{2}=\left\{ e\in E\mid e\text{ is not }\alpha_{0}\text{-nice}\right\} $
we have $\sum_{e\in E_{2}}c_{e}\leq\varepsilon$: Indeed, if this
is not the case we can replace $\left\{ C_{v}\right\} _{v\in V}$
with $\left\{ C_{v}'\right\} _{v\in V'}$ from the Lemma, which has
the same carrier and the same effective conductance ${\rm C_{eff}}\left(\rho\leftrightarrow\infty\right)$.\\
\uline{Step 2: construct a function \mbox{$f$} on the vertices:}
By the parabolicity of $\Omega$, by \ref{prop: Equivalent Condition for Parabolicity }
applied to the compact set with nonempty interior $K=P_{\rho}$, there
exists a Lipschitz function $\phi:\Omega\rightarrow\left[0,1\right]$
compactly supported in $\Omega$ with $\phi\mid_{P_{\rho}}\equiv1$
such that its Dirichlet energy satisfies:
\[
\iintop_{\Omega}\left\Vert \nabla\phi\right\Vert ^{2}dA<\frac{\sin\left(\frac{1}{2}\alpha_{0}\right)}{2C_{1}\left(\alpha_{0}\right)}\cdot\varepsilon,
\]
where $C_{1}\left(\alpha_{0}\right)$ is the constant from Lemma \ref{lem: integration on polygon}.
By Lemma \ref{lem: compact set intersects finite number of polygons},
the support of $\phi$ intersects a finite set of polygons $\left\{ P_{w}\right\} _{w\in W}$,
and denote their union by $A=\bigcup_{w\in W\backslash\rho}P{}_{w}$.
Notice that $\phi\mid_{\partial A}\equiv0$. Without loss of generality,
we may now assume that $\phi$ is harmonic in $int\left(A\right)\backslash P_{\rho}$:
If this is not the case, we can replace the values of $\phi$ there
by the unique harmonic solution $\psi$ to the Dirichlet problem with
boundary conditions $\psi\mid_{\partial A}\equiv0$ and $\psi\mid_{\partial P{}_{\rho}}\equiv1$
. This solution is indeed unique and minimizes the Dirichlet energy
due to the piecewise smoothness of the boundary (being a finite union
of polygonal lines). Notice that we may still assume that $0\leq\phi\leq1$
due to the maximum and minimum principles for harmonic functions.
Once $\phi$ is defined, it induces a function $f$ on $V$ by letting
$f\left(v\right)$ be the value of $\phi$ on the center of $C_{v}$.
\\
Note that under the extra assumption that the continuous function
$\phi$ is harmonic, the function $f$ defined here coincides with
the function defined in \cite{gurel2017recurrence} in the proof that
parabolicity of the domain implies recurrence of the simple random
walk.\\
\uline{Step 3: Bound the discrete energy of \mbox{$f$}:} Recall
that $E_{2}=\left\{ e\in E\mid e\text{ is not \ensuremath{\alpha_{0}}-nice}\right\} $,
and set $E_{1}=E\backslash E_{2}$. Then the Dirichlet energy of $f$
is:
\[
\mathcal{E}\left(f\right)=\underbrace{\sum_{uv\in E_{1}}c{}_{uv}\left(f\left(v\right)-f\left(u\right)\right)^{2}}_{\mathcal{E}_{1}}+\underbrace{\sum_{uv\in E_{2}}c{}_{uv}\left(f\left(v\right)-f\left(u\right)\right)^{2}}_{\mathcal{E}_{2}}.
\]
By using $0\leq f\leq1$ and the choice of $\alpha_{0}$, we can bound
the second summand:
\begin{equation}
\mathcal{E}_{2}\leq\sum_{uv\in E_{2}}c{}_{uv}\cdot1\leq\varepsilon\label{eq: bound on non-nice edges}
\end{equation}
For any other edge $e=uv\in E_{1}$, denote by $M$,$N$ the ends
of its dual edge, i.e the incenters of the two faces incident to $e$.
For any $a,b\in\mathbb{R}^{2}$ denote by $T_{ab}$ the straight line
segment connecting $a$ to $b$. For every $z\in T_{MN}$, we can
use the identity $\left(x+y\right)^{2}\leq2x^{2}+2y^{2}$ to obtain:
\[
c_{uv}\left(f\left(v\right)-f\left(u\right)\right)^{2}=c_{uv}\left(\phi\left(v\right)-\phi\left(u\right)\right)^{2}\leq2c{}_{uv}\left(\phi\left(v\right)-\phi\left(z\right)\right)^{2}+2c{}_{uv}\left(\phi\left(u\right)-\phi\left(z\right)\right)^{2}.
\]
Since this inequality holds for any $z\in T_{MN},$ it also holds
in expectation:
\[
c_{uv}\left(f\left(v\right)-f\left(u\right)\right)^{2}\leq\frac{2c_{uv}}{\left|T_{MN}\right|}\intop_{T_{MN}}\left(\phi\left(v\right)-\phi\left(z\right)\right)^{2}dz+\frac{2c_{uv}}{\left|T_{MN}\right|}\intop_{T_{MN}}\left(\phi\left(u\right)-\phi\left(z\right)\right)^{2}dz.
\]
We continue to bound the first summand above. Since $\phi$ is smooth
except maybe at the boundaries of polygons, we can write $\phi\left(v\right)-\phi\left(z\right)=\intop_{T_{vz}}\nabla\phi\left(\underline{r}\right)d\underline{r}$.
Thus, using the definition of $c_{uv}$ and Cauchy-Schwarz for line
integrals we bound:
\[
\frac{2c_{uv}}{\left|T_{MN}\right|}\intop_{T_{MN}}\left(\phi\left(v\right)-\phi\left(z\right)\right)^{2}dz=\frac{2}{\left|T_{uv}\right|}\intop_{T_{MN}}\left(\intop_{T_{vz}}\nabla\phi\left(\underline{r}\right)d\underline{r}\right)^{2}dz\leq2\intop_{T_{MN}}\frac{\left|T_{vz}\right|}{\left|T_{uv}\right|}\intop_{T_{vz}}\left\Vert \nabla\phi\left(\underline{r}\right)\right\Vert ^{2}drdz.
\]
Since the angles $\measuredangle vzM$ and $\measuredangle vzN$ add
up to $\pi$ radians, one of these angles is at least $\frac{\pi}{2}$.
Using this and $\left|T_{uv}\right|\geq r_{v}$ we bound:
\[
\frac{\left|T_{vz}\right|}{\left|T_{uv}\right|}\leq\max\left\{ \frac{\left|T_{vM}\right|}{r_{v}},\frac{\left|T_{vN}\right|}{r_{v}}\right\} =\max\left\{ \frac{1}{\sin\measuredangle vMN},\frac{1}{\sin\measuredangle vNM}\right\} \leq\frac{1}{\sin\left(\frac{1}{2}\alpha_{0}\right)},
\]
where the last inequality holds because $uv\in E_{1}$ and hence $T_{MN}$
is $\alpha_{0}$-nice for both $P_{u}$ and $P_{v}$. Using this and
a similar calculation for $u$ instead of $v$, we get:
\begin{align*}
c_{uv}\left(f\left(v\right)-f\left(u\right)\right)^{2} & \leq\frac{2}{\sin\left(\frac{1}{2}\alpha_{0}\right)}\left(\intop_{T_{MN}}\intop_{T_{vz}}\left\Vert \nabla\phi\left(\underline{r}\right)\right\Vert ^{2}drdz+\intop_{T_{MN}}\intop_{T_{uz}}\left\Vert \nabla\phi\left(\underline{r}\right)\right\Vert ^{2}drdz\right).
\end{align*}
For each $v\in V$ set $X_{v}$ to be the set of sides of $P_{v}$
that are dual to some edge $vu\in E_{1}$. Then summing up the last
inequality over all $uv\in E_{1}$ and using Lemma \ref{lem: Integration on parts of a polygon}
we bound:
\begin{align}
\mathcal{E}_{1} & \leq\frac{2}{\sin\left(\frac{1}{2}\alpha_{0}\right)}\cdot\sum_{v\in V}\sum_{S\in X_{v}}\left(\intop_{S}\intop_{T_{vz}}\left\Vert \nabla\phi\left(\underline{r}\right)\right\Vert ^{2}drdz\right)\leq\frac{2}{\sin\left(\frac{1}{2}\alpha_{0}\right)}\cdot\sum_{v\in V}C_{1}\left(\alpha_{0}\right)\iintop_{P_{v}}\left\Vert \nabla\phi\right\Vert ^{2}dA\leq\nonumber \\
 & \leq\frac{2C_{1}\left(\alpha_{0}\right)}{\sin\left(\frac{1}{2}\alpha_{0}\right)}\cdot\iintop_{\Omega}\left\Vert \nabla\phi\right\Vert ^{2}dA\leq\varepsilon.\label{eq: bound on energy of E2-1}
\end{align}
Combining inequalities \ref{eq: bound on non-nice edges} and \ref{eq: bound on energy of E2-1}
we can finally bound the Dirichlet energy of $f$:
\[
\mathcal{E}\left(f\right)=\mathcal{E}_{1}+\mathcal{E}_{2}\leq2\varepsilon.
\]
\uline{Step 5: Deducing recurrence:} We've found a function $f:V\rightarrow\mathbb{R}$
with $\mathcal{E}\left(f\right)\leq2\varepsilon$. Furthermore, since
$\phi$ vanishes outside a finite union of polygons, $f$ is finitely
supported. Lastly, $f\left(\rho\right)=1$. By Dirichlet's Principle,
the effective conductance in the network $\left(G,c\right)$ satisfies
${\rm C_{eff}}\left(\rho\leftrightarrow\infty\right)\leq2\varepsilon.$
Since $\varepsilon$ is arbitrary, ${\rm C_{eff}}\left(\rho\leftrightarrow\infty\right)=0$
and hence the network is recurrent.
\end{proof}

\section{Extension to Higher Dimensions\label{sec:Extension-to-higher}}

The goal of this section is to define a higher-dimensional analogue
of the Dubejko weights and prove Proposition \ref{prop:higher dimensions martingale theorem},
showing that the weighted random walk is a martingale.
\begin{defn}
Given a set of spheres $\left\{ S_{v}\right\} _{v\in V}$ in $\mathbb{R}^{3}$
with disjoint interiors, we can define its \textbf{tangency graph}:
The vertex set is $V$ and two vertices $v,u\in V$ are connected
by an edge iff their spheres $S_{v}$ and $S_{u}$ are tangent. Given
two neighbouring vertices $u,v\in V$, we define their \textbf{common
tangent plane} to be the unique plane in $\mathbb{R}^{3}$ that is
tangent to both $S_{v}$ and $S_{u}$. For each neighbour of some
$v\in V$, their common tangent plane cuts $\mathbb{R}^{3}$ into
two half-spaces, one of them containing $v$. If the intersection
over all the neighbours of $v$ of the half-spaces containing $v$
is bounded, we say that $v$ is \textbf{covered} and call the intersection
the \textbf{polyhedron of $v$} denoted by $P_{v}$. We say that $\left\{ S_{v}\right\} _{v\in V}$
is \textbf{covering} if all vertices are covered and for every edge
$uv\in E$ their respective polyhedra intersect along a common face.
\end{defn}

\begin{defn}
Let $\left\{ S_{v}\right\} _{v\in V}$ be a covering 3-sphere packing
with tangency graph $G=\left(V,E\right)$. Then in a similar fashion
to the Dubejko weights, we define the weight of an edge $uv\in E$
to be:
\begin{equation}
c_{uv}:=\frac{A_{uv}}{\left\Vert v-u\right\Vert },\label{eq: 3d weights}
\end{equation}
where $A_{uv}$ is the area of the shared face of $P_{u}$ and $P_{v}$.
\end{defn}

\begin{rem}
Similar definitions can be made in dimensions $d>3$, where instead
of the area we use the $d-1$-dimensional volume of the shared face
of the polytopes of $P_{u}$ and $P_{v}$, and the following Proposition
would remain true.
\end{rem}

\begin{prop}
\label{prop:higher dimensions martingale theorem}Let $\left\{ S_{v}\right\} _{v\in V}$
be a covering 3-sphere packing with tangency graph $G=\left(V,E\right)$.
Identify the vertices of $V$ with the centers of the their spheres,
and let $\left(Z_{n}\right)_{n\in\mathbb{N}}$ be the sequence of
centers of spheres visited in a random walk weighted according to
Equation \ref{eq: 3d weights} started at some vertex $v_{0}\in V$
. Then $\left(Z_{n}\right)_{n\in\mathbb{N}}$ is a martingale.
\end{prop}

\begin{proof}
Set $\pi\left(x\right)=\sum_{y\sim x}c_{xy}$ and denote the transition
probabilities by $P_{x,y}=\frac{c_{xy}}{\pi\left(x\right)}$. Let
$v\in V$, and let $u_{1},u_{2},...,u_{n}\in V$ be its neighbors
in $G$. It is enough to show:
\[
\sum_{i=1}^{n}P_{v,u_{i}}\cdot\left(u_{i}-v\right)=0.
\]
Plugging in the definition of the weights, we get:
\[
\sum_{i=1}^{n}P_{v,u_{i}}\cdot\left(u_{i}-v\right)=\frac{1}{\pi\left(v\right)}\sum_{i=1}^{n}A_{u_{i}v}\frac{u_{i}-v}{\left\Vert u_{i}-v\right\Vert }.
\]
Write $\partial P_{v}=\bigcup_{i=1}^{n}Q_{vu_{i}}$, where $Q_{vu_{i}}$
is the common face of $P_{v}$ and $P_{u_{i}}$, and notice that $\frac{u_{i}-v}{\left\Vert u_{i}-v\right\Vert }$
is the outward-pointing normal of $Q_{vu_{i}}$. Thus, we can write:
\[
\sum_{i=1}^{n}A_{u_{i}v}\frac{u_{i}-v}{\left\Vert u_{i}-v\right\Vert }=\sum_{i=1}^{n}Area\left(Q_{vu_{i}}\right)\cdot\hat{n}_{Q_{vu_{i}}}.
\]
We wish to show that this vector is zero. Let $\hat{m}$ be any unit
vector, then using the divergence theorem and the fact that the divergence
of a constant vector field is zero we find:
\begin{align*}
\hat{m}\cdot\left(\sum_{i=1}^{n}Area\left(Q_{vu_{i}}\right)\cdot\hat{n}_{Q_{vu_{i}}}\right) & =\oiintop_{\partial P_{v}}\hat{m}\cdot dS=\iiintop_{P_{v}}\left(\nabla\cdot\hat{m}\right)dV=0.
\end{align*}
 Since this is true for any unit vector $\hat{m}$, we have $\sum_{i=1}^{n}A_{u_{i}v}\frac{u_{i}-v}{\left\Vert u_{i}-v\right\Vert }=0$
as needed.
\end{proof}
\begin{rem}
The second part of the last proof is simply the fact that the vector
area of any closed surface is zero. A physical interpretation of this
claim is the following: Fill a metallic shell in the shape of $\partial P_{v}$
with water and put it somewhere far away in outer space. The force
acting on each face due to the water pressure is proportional to its
area and acts outwards. Hence, $\sum_{i=1}^{n}Area\left(Q_{vu_{i}}\right)\cdot\hat{n}_{Q_{vu_{i}}}=0$
is exactly the claim that the net force acting on the shell is zero
- or in other words that the shell would not start accelerating spontaneously
in some direction.

\noindent \textbf{Acknowledgements.} The authors wish to thank Asaf
Nachmias for many useful discussions and comments.

\bibliographystyle{plain}
\bibliography{arxiv_version}

\begin{thebibliography}{10}

\bibitem{axler2013harmonic}
Sheldon Axler, Paul Bourdon, and Ramey Wade.
\newblock {\em Harmonic function theory}, volume 137.
\newblock Springer Science \& Business Media, 2013.

\bibitem{coxeter1969introduction}
Harold Scott~Macdonald Coxeter, Harold Scott~Macdonald Coxeter, Harold
  Scott~Macdonald Coxeter, and Harold Scott~Macdonald Coxeter.
\newblock {\em Introduction to geometry}, volume 136.
\newblock Wiley New York, 1969.

\bibitem{dubejko1997random}
Tomasz Dubejko.
\newblock Random walks on circle packings.
\newblock {\em Contemporary Mathematics}, 211:169--182, 1997.

\bibitem{duffin1968potential}
Richard~James Duffin.
\newblock Potential theory on a rhombic lattice.
\newblock {\em Journal of Combinatorial Theory}, 5(3):258--272, 1968.

\bibitem{durrett2010probability}
Rick Durrett.
\newblock {\em Probability: theory and examples}.
\newblock Cambridge university press, 2010.

\bibitem{grigor1999analytic}
Alexander Grigor'yan.
\newblock Analytic and geometric background of recurrence and non-explosion of
  the brownian motion on riemannian manifolds.
\newblock {\em Bulletin of the American Mathematical Society}, 36(2):135--249,
  1999.

\bibitem{gurel2017recurrence}
Ori Gurel-Gurevich, Asaf Nachmias, and Juan Souto.
\newblock Recurrence of multiply-ended planar triangulations.
\newblock {\em Electronic Communications in Probability}, 22, 2017.

\bibitem{he1993fixed}
Zheng-Xu He and Oded Schramm.
\newblock Fixed points, koebe uniformization and circle packings.
\newblock {\em Annals of Mathematics}, 137(2):369--406, 1993.

\bibitem{he1995hyperbolic}
Zheng-Xu He and Oded Schramm.
\newblock Hyperbolic and parabolic packings.
\newblock {\em Discrete \& Computational Geometry}, 14(2):123--149, 1995.

\bibitem{kemeny2012denumerable}
John~G Kemeny, J~Laurie Snell, and Anthony~W Knapp.
\newblock {\em Denumerable Markov chains: with a chapter of Markov random
  fields by David Griffeath}, volume~40.
\newblock Springer Science \& Business Media, 2012.

\bibitem{koebe1936kontaktprobleme}
Paul Koebe.
\newblock {\em Kontaktprobleme der konformen Abbildung}.
\newblock Hirzel, 1936.

\bibitem{lyons2017probability}
Russell Lyons and Yuval Peres.
\newblock {\em Probability on trees and networks}, volume~42.
\newblock Cambridge University Press, 2017.

\bibitem{nachmias2018planar}
Asaf Nachmias.
\newblock Planar maps, random walks and circle packing.
\newblock {\em arXiv preprint arXiv:1812.11224}, 2018.

\bibitem{rodin1987convergence}
Burt Rodin and Dennis Sullivan.
\newblock The convergence of circle packings to the riemann mapping.
\newblock {\em Journal of Differential Geometry}, 26(2):349--360, 1987.

\bibitem{stephenson2005introduction}
Kenneth Stephenson.
\newblock {\em Introduction to circle packing: The theory of discrete analytic
  functions}.
\newblock Cambridge University Press, 2005.

\bibitem{woess2000random}
Wolfgang Woess.
\newblock {\em Random walks on infinite graphs and groups}, volume 138.
\newblock Cambridge university press, 2000.

\end{thebibliography}
\end{rem}

\end{document}